\documentclass[11pt, reqno, a4paper, final]{amsart}

\usepackage[applemac]{inputenc}      
\usepackage[english]{babel}
\usepackage[T1]{fontenc}
\usepackage{amsmath,amsfonts,amssymb,amsthm,amscd}
\usepackage{eucal}
\usepackage{array}
\usepackage{color}
\usepackage{bbm}
\usepackage{graphicx}
\usepackage{hyperref}
\usepackage{a4wide}										 	
\usepackage{mathrsfs}									  
\usepackage[all]{xy}										
\usepackage{url}												
\usepackage{verbatim,epsfig}
\usepackage[notref,notcite,draft]{showkeys}   
\usepackage{xfrac} 
\usepackage{extarrows}

\newcounter{claim-counter}

\theoremstyle{plain}

\newtheorem{thm}{Theorem}[section]

\newtheorem{lem}[thm]{Lemma}
\newtheorem{lemma}[thm]{Lemma}
\newtheorem{prop-defi}[thm]{Definition \& Proposition}
\newtheorem{prop}[thm]{Proposition}
\newtheorem{por}[thm]{Porism}
\newtheorem*{thm*}{Theorem}
\newtheorem*{prop*}{Proposition}
\newtheorem*{cor*}{Corollary}
\newtheorem{proposition}[thm]{Proposition}

\usepackage{thmtools}
\declaretheorem[style=theorem,name={Theorem}]{theoremletter}

\theoremstyle{definition}
\newtheorem{defi}[thm]{Definition}
\newtheorem{definition}[thm]{Definition}

\newtheorem{ex}[thm]{Example}

\newtheorem{rem}[thm]{Remark}
\newtheorem{remark}[thm]{Remark}
\newtheorem{claim}[claim-counter]{Claim}
\newtheorem*{claim*}{Claim}

\newcommand{\NN}{{\mathbb N}}
\newcommand{\ZZ}{{\mathbb Z}}

\newcommand{\RR}{{\mathbb R}}
\newcommand{\CC}{{\mathbb C}}

\newcommand{\Z}{{\mathcal Z}}
\renewcommand{\H}{\mathcal{H}}

\newcommand{\R}{{\mathcal R}}

\newcommand{\C}{{\mathscr C}}

\newcommand{\B}{{\mathscr B}}

\renewcommand{\max}{{\operatorname{max}}}

\newcommand{\ip}[2]{\left\langle {#1}\hspace{0.05cm}, \hspace{0.05cm}{#2}\right \rangle}

\newcommand{\varps}{{\varepsilon}}

\newcommand{\To}{\longrightarrow}

\newcommand{\id}{\operatorname{id}}

\renewcommand{\leq}{\leqslant}
\renewcommand{\geq}{\geqslant}

\newcommand{\bbb}{{\mathbbm{1}}}

\newcommand{\covol}{{\operatorname{covol}}}
\renewcommand{\S}{{\mathcal{S}}}
\renewcommand{\restriction}{\rvert}

\allowdisplaybreaks

\title[Measure equivalence and coarse equivalence for  unimodular lc groups]{Measure equivalence and coarse equivalence for unimodular locally compact  groups}

\author{Juhani Koivisto, David Kyed and Sven Raum}
\address{Juhani Koivisto, Department of Mathematics and Computer Science, University of Southern Denmark, Campusvej 55, DK-5230 Odense M, Denmark}
\email{koivisto@imada.sdu.dk}

\address{David Kyed, Department of Mathematics and Computer Science, University of Southern Denmark, Campusvej 55, DK-5230 Odense M, Denmark}
\email{dkyed@imada.sdu.dk}

\address{Sven Raum,
Department of Mathematics, Stockholm University, SE-106 91 Stockholm, Sweden}
\email{raum@math.su.se}


\begin{document}

\begin{abstract}
This article is concerned with measure equivalence and uniform measure equivalence of locally compact, second countable groups. We show that two unimodular, locally compact, second countable groups are measure equivalent if and only if they admit free, ergodic, probability measure preserving actions whose cross section equivalence relations are stably orbit equivalent.  Using this we  prove that in the presence of amenability any two  such groups are measure equivalent and that both amenability and property (T) are preserved under measure equivalence,  extending results of Connes-Feldman-Weiss and Furman. Furthermore, we introduce a notion of uniform measure equivalence for  unimodular, locally compact, second countable groups,  and prove that under the additional assumption of amenability this notion coincides with coarse equivalence, generalizing results of Shalom and Sauer. Throughout the article we rigorously treat measure theoretic issues arising in the setting of non-discrete groups. 
\end{abstract}

\maketitle

\section{Introduction}

Measure equivalence for countable discrete groups was originally introduced by Gromov \cite{gromov-asymptotic-invariants} as a measurable analogue of quasi-isometry and has since then proven to be an important tool in geometric group theory with connections to ergodic theory and operator algebras. Notably, measure equivalence was used by Furman in \cite{furman-gromovs-measure-equivalence, Furman-OE-rigidity} to prove strong rigidity results for lattices in higher rank simple Lie groups,  {and continuing this line of investigation}, Bader, Furman and Sauer \cite{BFS-integrable} introduced measure equivalence in the setting of unimodular\footnote{See also \cite{deprez-li-ME}  and the \cite{KKR18} for the non-unimodular case.}, locally compact, second countable groups.
The first aim of the present paper is to establish a rigorous understanding of this notion of measure equivalence and its relationship to other established notions of equivalence between locally compact, second countable groups, such as orbit equivalence of probability measure preserving actions and stable orbit equivalence of cross section equivalence relations associated with such actions.  We obtain proofs of results that might be considered folklore by parts of the community, but the present piece of work has the virtue of being largely self-contained and furthermore offers a careful treatment of the subtle measure theoretic issues arising  from the fact that the saturation of a null set (respectively Borel set) with respect to a non-discrete locally compact group is not necessarily null (respectively Borel).  It is precisely these measurability issues  that form the gap between folklore results and the rigorous treatment offered here.
Our first main theorem establishes an equivalence between the notions of measure equivalence and (stable) orbit equivalence 
within the framework of unimodular, locally compact, second countable groups.
{For discrete groups the equivalence of (i) and (ii) below can be found in \cite[Proposition 1.22]{carderi-maitre}; and the equivalence of (i) and (iii) for discrete groups can be found in \cite[Theorem 2.3]{gaboriau} (see also \cite[Lemma 3.2 and Theorem 3.3]{Furman-OE-rigidity} for related earlier results).
\begin{theoremletter}[Theorem \ref{thm:me-and-soe-equivalence}]  \label{thm:me-oe-soe}
For unimodular, locally compact, second countable groups $G$ and $H$ the following are equivalent.
\begin{itemize}
\item[(i)] $G$ and $H$ are measure equivalent.
\item[(ii)] $G \times \mathrm S^1$ and $H \times \mathrm S^1$ admit orbit equivalent, essentially free, ergodic, probability measure preserving actions on standard Borel probability spaces.
\item[(iii)] $G$ and $H$ admit essentially free, ergodic, probability measure preserving  actions on standard Borel probability spaces whose cross section equivalence relations
  are stably orbit equivalent.
\end{itemize}
\end{theoremletter}
Here, and in what follows, $\mathrm S^1$ denotes the circle group and we refer to Sections \ref{sec:me} and \ref{sec:me-oe} for definitions of the various notions appearing in the statement and for remarks on item (ii).  Note that in case both $G$ and $H$ are  non-discrete, the amplification by $\mathrm S^1$ in statement (ii) is not necessary;  see Theorem \ref{thm:me-and-soe-equivalence} for a precise statement.  A particular instance of the rigorous treatment provided by this work is Theorem \ref{thm:strict-coupling}, worth mentioning in its own right, which shows that one can always pass from a measure equivalence coupling to a strict measure equivalence coupling; i.e.~one where all maps involved are Borel isomorphisms and genuinely equivariant. \\

An important property of measure equivalence for discrete groups is the fact that all countably infinite, amenable discrete groups are pairwise measure equivalent, which follows from the work by Ornstein and Weiss in \cite{ornstein-weiss} as 
proven by Furman in \cite{furman-gromovs-measure-equivalence}.
As a consequence of Theorem \ref{thm:me-oe-soe} we deduce a similar result in the non-discrete setting.

\begin{theoremletter}[Theorems \ref{thm:amenability-ME-invariant} and \ref{thm:amenable-groups-are-ME}]\label{thm:amenable-grps-are-me-intro}
  All non-compact, amenable, unimodular, locally compact, second countable groups are pairwise measure equivalent.  Conversely, 
if $G$ and $H$ are measure equivalent, unimodular, locally compact, second countable groups and one of them is amenable then so is the other.
\end{theoremletter}

Further, we clarify the role of measure equivalence  in connection with property (T) by extending \cite[Corollary 1.4]{furman-gromovs-measure-equivalence} to the locally compact case, again carefully treating measure theoretic issues, by passing to the aforementioned strict coupling provided by Theorem \ref{thm:strict-coupling}.
\begin{theoremletter}\label{thm:me-and-property-T}
If $G$ and $H$ are measure equivalent, unimodular, locally compact, second countable groups and one has property \emph{(T)} then so does the other.
\end{theoremletter}



Uniform measure equivalence for finitely generated, countable discrete groups was introduced by Shalom in \cite{shalom-rigidity}, combining the notions of measure equivalence and quasi-isometry, with the purpose of capturing the important situation of two cocompact, finitely generated lattices in a common Lie group.  
With  the notion of measure equivalence extended to the class of unimodular, locally compact, second countable groups, it is natural to also extend uniform measure equivalence to this setting, in such a way that if two such groups are closed, cocompact subgroups of the same locally compact, second countable group then they are uniformly measure equivalent.  We introduce such a notion in Definition \ref{def:UME},  show that the situation just described indeed gives rise to uniformly measure equivalent groups in Proposition \ref{prop:cocompact-ume} and prove that our definition reduces to the existing one \cite[Definition 2.23]{sauer-thesis}\footnote{In \cite{sauer-thesis} the term \emph{bounded} measure equivalence is used.} in the case of finitely generated, discrete groups.
For finitely generated, discrete, amenable groups it was shown by  Shalom \cite{shalom-harmonic-analysis} and Sauer \cite{sauer-thesis}
that uniform measure equivalence coincides with quasi-isometry {(equivalently coarse equivalence; cf.~\cite[Corollary 4.7]{yves-survey})} and we show that this result also carries over to the locally compact setting by means of the following theorem.

\begin{theoremletter}[Theorem \ref{thm:uniform-me-amenable-groups}]\label{thm:ume-equals-ce-for-amenable}
Two amenable, unimodular, locally compact, second countable groups are uniformly measure equivalent if and only if they are coarsely equivalent. 
\end{theoremletter}

In addition to the introduction, this article has five sections. 
In Section \ref{sec:me} we introduce measure equivalence after Bader-Furman-Sauer \cite{BFS-integrable}, paying particular attention to measure theoretic aspects, and in Section \ref{sec:me-oe} we prove Theorem \ref{thm:me-oe-soe}.  In Section \ref{sec:me-amenable-groups} we use Theorem~\ref{thm:me-oe-soe} in order to show how amenability behaves with respect to measure equivalence as stated in Theorem \ref{thm:amenable-grps-are-me-intro}, likewise in Section \ref{sec:T} we show that property (T) is invariant under measure equivalence of (unimodular) locally compact second countable groups.  In the final  Section \ref{sec:uniform-me}, we provide a definition of uniform measure equivalence, prove Theorem \ref{thm:ume-equals-ce-for-amenable} and check several statements that are expected to hold true for a good notion of uniform measure equivalence.

\subsection*{Notation for locally compact groups}\label{sec:notation}
All locally compact groups are assumed to be Hausdorff and we will abbreviate `locally compact second countable' by `lcsc'.
Moreover,  all group actions on spaces will be left actions.

\subsection*{Acknowledgments} The authors would like to thank Uri Bader, Yves Cornulier, Damien Gaboriau and Christian Rosendal for interesting comments and suggestions, and the anonymous referees for suggesting a number of improvements and for pointing out a mistake in a previous version of the paper.
DK and JK gratefully acknowledge the financial support from  the Villum Foundation (grant no.~7423) and the DFF-Research project  \emph{Automorphisms and Invariants of Operator Algebras} (grant no.~7014-00145B).
SR acknowledges the financial support from Vetenskapsr{\r{a}}det's Starting Grant no. 2018-04243.  DK and SR would like to thank the Isaac Newton Institute for Mathematical Sciences for support and hospitality during the programme {\it Operator Algebras: Subfactors and their Applications} (supported by EPSRC Grant Number EP/K032208/1) where work on this paper was undertaken.

\section{Measure equivalence couplings}
\label{sec:me}

In this section we fix our measure theoretical notation and recall relevant facts about measure equivalence couplings between unimodular lcsc  groups.  Furthermore, we prove  that any measure equivalence coupling can be replaced by a strict coupling.  In order to be able to treat the measure theoretical aspects at a sufficiently rigid level, we begin by defining our notion of measurable spaces and group actions on such.

\begin{defi}
  \label{def:measurable-sapce}
By a \emph{measurable space} we mean a set $X$ endowed with a $\sigma$-algebra $\B$ {whose elements are  called the \emph{measurable subsets} of $X$}.  If $X$ is furthermore endowed with a measure $\mu \colon \B \to [0,\infty]$, the triple $(X,\B,\mu)$ is referred to as a \emph{measure space}, and when the $\sigma$-algebra is clear from the context we will {often suppress it notationally and} simply write $(X,\mu)$.  We will use the standard measure theoretic lingo and refer to a subset $N\subset X$ as a \emph{null set} if $N$ is contained in a measurable subset of measure zero. Similarly, the complement of a null set will be referred to as being \emph{conull} and any non-null subset is referred to as \emph{non-negligible}. 
If $\mu$ and $\nu$ are two measures on $\B$ with the same null sets we write $\mu\sim \nu$ and refer to the two measures as being \emph{equivalent}, and we furthermore denote by $[\mu]$ the \emph{measure class} of $\mu$; i.e.~the set of all measures equivalent to it.
An isomorphism of measurable spaces $(X,\B)$ and $(Y,\C)$ is a bijective measurable map $f\colon X \to Y $ whose inverse is also measurable.    If $(X, \B)$ and $(Y, \C)$ are moreover endowed with  measure classes  $[\mu]$ and $[\nu]$ of  measures $\mu$ and $\nu$, respectively, a \emph{non-singular isomorphism} between $(X, \B, [\mu])$ and $(Y, \C, [\nu])$ is a measurable map $f: X \to Y$ for which there are conull, measurable subsets $X_0 \subset X$ and $Y_0 \subset Y$ such that $f$ restricts to a measurable isomorphism $f_0\colon X_0 \to Y_0$ and $f_*([\mu]) = [\nu]$.  An \emph{isomorphism} of measure spaces $(X,\B,\mu)$ and $(Y,\C, \nu)$ is a measurable map $f\colon X\to Y$ for which there exist conull, measurable subsets $X_0\subset X$ and $Y_0\subset Y$ such that $f$ restricts to an isomorphism of measurable spaces $f_0\colon X_0\to Y_0$ and  $f_*(\mu)=\nu$. Lastly, for a topological space $X$ the $\sigma$-algebra generated by the open sets is referred to as the \emph{Borel $\sigma$-algebra} and its sets are called \emph{Borel sets}.
\end{defi}

\begin{remark}\label{rem:borel-vs-non-borel} 
  The definition of  a non-singular isomorphism  between  spaces $(X, [\mu])$ and $(Y, [\nu])$ is required to be compatible at all levels with the measurable structures of $X$ and $Y$.  This is no restriction of generality when compared with other possible notions.  Indeed, let $X_0 \subset X$ and $Y_0 \subset Y$ be two not necessarily measurable conull subsets and let $f\colon X_0 \rightarrow Y_0$ a measurable isomorphism for the restricted $\sigma$-algebras that satisfies $f_*[\bar{\mu}\restriction_{X_0}] = [\bar{\nu}\restriction_{Y_0}]$  where $\bar{\mu}$ and $\bar{\nu}$ denote the completed measures.  Since $X_0 \subset X$ is conull, there is a measurable conull subset $X_{00} \subset X_0$.  Then also $Y_{00} := f(X_{00})$ is conull in $Y_0$  -- and hence in $Y$  -- so there is a measurable subset $Y_{000} \subset Y_{00}$ which is conull in $Y$.  We set $X_{000} := f^{-1}(Y_{000})$, which is measurable with respect to the relative $\sigma$-algebra of $X_{00}$.  Since $X_{00}$ is measurable in $X$, it follows that $X_{000}$ is measurable as a subset of $X$, and  it is  conull by construction.  We can therefore restrict $f$ to a measurable isomorphism $X_{000} \rightarrow Y_{000}$ and extend this to a measurable map from $X$ to  $Y$, which is an isomorphism of non-singular spaces in the sense of Definition \ref{def:measurable-sapce}.  For technical reasons, this definition is preferable to superficially more general ones.  A similar reasoning applies to isomorphisms between measure spaces.
\end{remark}

\begin{defi}
  Let $(X,\B)$ be a measurable space and $G$ be an lcsc group. Then $X$ is said to be a \emph{measurable $G$-space} if its endowed with an action of $G$ for which the action map $G\times X \to X$ is measurable; here $G$ is equipped with the Borel $\sigma$-algebra and $G\times X$ with the product $\sigma$-algebra.   Moreover, if the measurable $G$-space $X$ is endowed with  a  measure $\mu$ and the $G$-action is \emph{non-singular}, i.e.~$[g_*\mu]=[\mu]$ for all $g\in G$, then $(X, \B, \mu)$ is called a \emph{non-singular $G$-space}.  
 If the action is actually \emph{measure preserving}, i.e.~$g_*\mu=\mu$ for all $g\in G$,   then $(X,\B,\mu)$ is called a \emph{measure $G$-space}. 
Lastly,  an action preserving a given probability measure is referred to as a \emph{pmp action}.
\end{defi}

\begin{rem}
If $(X,\B)$ is a measurable space and $\eta$ is a $\sigma$-finite measure on it, then there exists a probability measure $\mu$ on $X$ which is equivalent to $\eta$; if $E_n\subset X$ are disjoint, measurable, of positive and finite $\eta$-measure and with conull union in $X$, one may for instance take the measure
\[
\mu(A):=\sum_{n=1}^\infty 2^{-n} \frac{\eta(A\cap E_n)}{\eta(E_n)}, \ A\in \B.
\]
This standard fact will be used repeatedly throughout the paper --- in particular we shall apply it to the Haar measures on a given lcsc group, which are all $\sigma$-finite since every lcsc group is $\sigma$-compact.

\end{rem}

Recall  that a topological space is said to be \emph{Polish} if it is homeomorphic to a complete, separable metric space and that a measurable space $(X,\B)$ is said to be a \emph{standard Borel space} if $X$ admits a topology with respect to which it is Polish and $\B$ is the Borel $\sigma$-algebra; i.e.~the $\sigma$-algebra generated by the open subsets. In this situation, elements of $\B$ are often referred to as \emph{Borel sets} and a measure defined on the Borel $\sigma$-algebra is referred to as a \emph{Borel measure}.  As the following two results show, the class of standard Borel spaces possesses a number of pleasant features.  Lemma \ref{lem:injective-borel}, which follows from an important theorem of Suslin, will be used repeatedly throughout the paper, while Theorem \ref{thm:lusin-novikov} will only be used in Section \ref{sec:me-oe}.

\begin{lem}[see  {\cite[Corollary 13.4 \& 15.2]{kechris-book}}]
  \label{lem:injective-borel}
Any Borel subset of a standard Borel space $X$ is itself standard Borel. Moreover, if $f\colon X\to Y$ is an injective, measurable map between standard Borel spaces, its image $f(X)$ is again standard Borel and  the inverse map $f^{-1}\colon f(X)\to X $ is measurable. 
\end{lem}
 
\begin{thm}[Lusin-Novikov, see {\cite[Theorem 18.10]{kechris-book}}]
  \label{thm:lusin-novikov}
  Let $X, Y$ be standard Borel spaces and let $E \subset X \times Y$ be a Borel subset such that $\pi_X^{-1}(x)$ is countable for every $x \in X$.  Then there is a countable partition into Borel subsets $E = \bigcup_{n \in \mathbb{N}} E_n$ such that $\pi_X|_{E_n}$ is injective for all $n \in \mathbb{N}$. 
\end{thm}



 In what follows,  we shall thus refer to a bijective measurable  map between standard Borel spaces as a  \emph{Borel isomorphism}, and the adjective `standard Borel' will always indicate that the underlying space is a standard Borel space. Thus, a \emph{standard Borel $G$-space} is a standard Borel space which is also a measurable $G$-space, and a \emph{standard Borel measure $G$-space} is a standard Borel $G$-space equipped with a {Borel} measure with respect to which the action is measure preserving. Lastly, a standard Borel measure space {with a Borel probability measure} will be referred to as a \emph{standard Borel probability space}.

\begin{defi}[{\cite{zimmer-book}}]
If $(X,\mu)$ and $(Y,\nu)$ are measure $G$-spaces, then $X$ and $Y$ are said to be  \emph{isomorphic} if there exist conull, $G$-invariant measurable subsets $X_0\subset  X$ and $Y_0\subset  Y$ and a $G$-equivariant isomorphism of measurable spaces $f\colon X_0\to Y_0$ such that $f_*\mu\restriction_{X_0}=\nu\restriction_{Y_0}$. As is standard,
we will often write $f\colon X\to Y$ in this situation.
\end{defi}

With these preliminaries taken care of, we can now introduce the notion of measure equivalence following Bader-Furman-Sauer.

\begin{defi}[{\cite{BFS-integrable}}]\label{def:me}
  Two  unimodular lcsc groups $G$ and $H$ with a choice of Haar measures $\lambda_G$ and $\lambda_H$ are said to be \emph{measure equivalent} if there exist a standard Borel measure $G\times H$-space $(\Omega,\eta)$  and two standard Borel measure spaces $(X,\mu)$ and $(Y,\nu)$ such that: 
  \begin{itemize}
  \item[(i)] both $\mu$ and $\nu$ are finite measures and $\eta$ is non-zero;
  \item[(ii)] there exists an isomorphism of measure $G$-spaces  $i\colon (G\times Y, \lambda_G\times \nu) \To(\Omega,\eta)$, where $\Omega$ is considered a measure $G$-space for the restricted action and $G\times Y$ is considered a measure $G$-space for the action $g.(g',y)=(gg',y)$;

  \item[(iii)] there exists an isomorphism of measure $H$-spaces $j:(H\times X, \lambda_H\times \mu) \to (\Omega,\eta)$, where $\Omega$ is considered a measure $H$-space for the restricted action and $H\times X$ is considered a measure $H$-space for the action $h.(h',x)=(hh',x)$.
  \end{itemize}
  A standard Borel space $(\Omega,\eta)$ with these properties is called a \emph{measure equivalence coupling} between $(G, \lambda_G)$ and $(H, \lambda_h)$, and whenever needed we will specify the additional data by writing  $(\Omega,\eta, X,\mu, Y,\nu, i,j)$.
\end{defi}

\begin{rem}
  Since the Haar measure on a locally compact group is unique up to scaling, the existence of a measure equivalence is  independent of the choice of Haar measures on the groups in question.  For a proof that measure equivalence is indeed an equivalence relation, the reader is referred to \cite[Appendix A]{BFS-integrable} where the composition of measure equivalence couplings is discussed; for the details, see also \cite[Appendix B]{zimmer-book} and \cite{mackey-point-realization}.
 Given a measure equivalence coupling  $(\Omega,\eta, X,\mu, Y,\nu, i,j)$ between two unimodular lcsc groups with Haar measure, $(G, \lambda_G)$ and $(H, \lambda_H)$, one can define its coupling constant as $\mu(X)/\nu(Y)$ and observe that it is multiplicative under composition and compatible with scaling of Haar measures.  So the \emph{fundamental group} of a unimodular lcsc group can be defined as the group of coupling constants of its self-couplings for a fixed choice of Haar measure.
By convention, two unimodular lcsc groups $G$ and $H$ are called \emph{measure equivalent} if there is a measure equivalence coupling $(\Omega,\eta, X,\mu, Y,\nu, i,j)$ between $(G, \lambda_G)$ and $(H, \lambda_H)$ for some choice of Haar measures.  The previous discussion makes it clear that, without loss of generality, one may assume that $\mu$ and $\nu$ are probability measures. Also note that, under the hypotheses in Definition \ref{def:me}, $(\Omega,\eta)$ is automatically $\sigma$-finite and
if $A\subset  G\times Y$ is a $G$-invariant subset then $A=G\times Y_0$ where $Y_0=\{y\in Y \mid \exists g\in G: (g,y)\in A\}$. Thus, the requirement on the map $i$ amounts to saying that there exist a measurable conull subset $Y_0\subset  Y$ and a measurable conull $G$-invariant subset $\Omega_0\subset \Omega$ such that $i\colon G\times Y_0 \to \Omega_0$ is a measure preserving, $G$-equivariant isomorphism of Borel spaces (and similarly for $j$).
\end{rem}

A recurring issue when working with group actions on measure spaces is that we often encounter only near actions in the sense of the following definition.
\begin{defi}
  A  measure preserving \emph{near action} of a lcsc group $G$ on a measure space $(X,  \mu)$ is a measurable map $G \times X \To X, (g,x)\mapsto g.x$, such that:
\begin{itemize}
\item[(i)] $e.x=x$ for almost all $x \in X$;
\item[(ii)] for all $g_1,g_2\in G$: $(g_1g_2).y=g_1.(g_2.y)$ for almost all $x\in X$;
\item[(iii)] For all $g\in G$ and all measurable subsets $A\subset X$ one has  $\mu(g.A)=\mu(A)$.
\end{itemize}
\end{defi}

Appealing  to a result of Mackey, we may actually replace a near action by a measure preserving action which agrees with the near action almost everywhere as long as the the group is second countable and the underlying Borel space is standard.
\begin{lemma}[\mbox{\cite[Lemma 3.2]{ramsey}, cf. \cite[Theorem 1]{mackey-point-realization}}]
  \label{lem:near-action-to-Borel-action}
  Let $(X, \mu)$ be a standard Borel probability space and let $G \times X \to X$ be a measure preserving near action of a lcsc group.  Then there is a standard Borel measure $G$-space $Y$ and an isomorphism of measure spaces $f\colon X \to Y$ such that for all $g \in G$ the equality $g.f(x) = f(g.x)$ holds for almost all $x \in X$.  
\end{lemma}

Recall that if $G$  and $H$ are topological groups and $(X,\mu)$ is a measure $G$-space then a  measurable map $\omega\colon G\times X \to H$ is called a \emph{measurable cocycle} if for all $g_1,g_2\in G$ there exists a conull subset $X_0\subset X$ such that the \emph{cocycle relation} $\omega(g_1g_2,x)=\omega(g_1,g_2.x)\omega(g_2,x)$ holds for all $x\in X_0$.  A measurable cocycle is said to be \emph{strict} if the cocycle relation is satisfied for all $x\in X$.  
Assume now that $G$ and $H$ are measure equivalent, unimodular, lcsc groups and denote by $(\Omega,\eta)$ a measure equivalence coupling with associated finite {standard Borel} measure spaces $(X,\mu)$ and $(Y,\nu)$ and measure space isomorphisms $i$ and $j$.  That is, there exist conull, measurable subsets $X_0\subset  X$, $Y_0\subset Y$, $\Omega_0,\Omega_0'\subset \Omega$ such that $\Omega_0$ and $\Omega_0'$ are $G$- and $H$-invariant, respectively, and the restrictions
\begin{align*}
i_0&:=i\restriction_{G\times Y_0} \colon G\times Y_0 \To \Omega_0\\
j_0&:=j\restriction_{H\times X_0} \colon H\times X_0 \To \Omega_0'
\end{align*}
are $G$- and $H$-equivariant, measure preserving, measurable isomorphisms, respectively. 
Then for every $h\in H$ there exists a conull, measurable subset $Y_h\subset Y_0$ such that for all $y\in Y_h$  and all $g\in G$, one has $h.i_0(g,y)\in \Omega_0$, and thus $i_0^{-1}(h.i_0(g,y))\in G\times Y_0$ makes sense. For such $y\in Y_h$ we now define $\omega_G({h},y)^{-1}\in G$ to be the  $G$-coordinate of   $i_0^{-1}(h.i_0({e_G},y))$ and $h.y$ to be the $Y$-coordinate of $i_0^{-1}(h.i_0({e_G},y))$.  Extending $h\colon Y_h\to Y $ to $Y$ by setting $h(y)=y$ for $y\in Y\setminus Y_h$, a direct computation shows that this defines a measure preserving near action.  

Another direct computation reveals that any {measurable}  extension of $\omega_G$ to $H\times Y$ is a measurable cocycle.
Note also that if $i$ itself were a measurable isomorphism and  $G$-equivariant at every point, then the associated near action of $H$ on $(Y,\nu)$ is a genuine action and the associated measurable cocycle $\omega_G$ is automatically strict.
In exactly the same manner we get a measure preserving near action $G\curvearrowright (X,\mu)$ with associated measurable cocycle $\omega_H\colon G\times X\to H$. 
\\

The following theorem shows that one may always replace a measure equivalence coupling with one where all the defining properties are satisfied pointwise, in contrast to almost everywhere.

\begin{thm}\label{thm:strict-coupling}
Let $(\Omega, \eta, X,\mu, Y,\nu, i , j)$ be a measure equivalence coupling between unimodular lcsc groups $G$ and $H$ with action $(g,h,t)\mapsto (g,h).t$.  Then there are conull, Borel subsets $\Omega' \subset \Omega$, $X' \subset X$ and $Y' \subset Y$ and a measure preserving action $G\times H\times \Omega' \To \Omega'$, $(g,h,t)\mapsto (g,h)\triangleright t$, such that the following hold:
\begin{itemize}
\item[(i)] for all $(g,h)\in G\times H$ one has $(g,h).t=(g,h)\triangleright t$ for $\eta$-almost all $t\in \Omega'$;
\item[(ii)] there exists a Borel isomorphism  $i'\colon G\times Y' \to \Omega'$  with $i'_*(\lambda_G\times \nu\restriction_{Y'})=\eta\restriction_{\Omega'}$ and $i'(g_0g, y)= (g_0,e_H)\triangleright i'(g,y)$ for all $g_0,g\in G$ and all $y\in Y'$;
\item[(ii)] there exists a Borel isomorphism  $j'\colon H\times X' \to \Omega'$  with $j'_*(\lambda_H\times \mu\restriction_{X'})=\eta\restriction_{\Omega'}$ and $j'(h_0h, x)= (e_G,h_0)\triangleright j'(h,x)$ for all $h_0,h\in H$ and all $x\in X'$.
\end{itemize}

\end{thm}

\begin{proof}
By definition, there exist conull, Borel subsets $X_0\subset X$, $Y_0\subset Y$, $\Omega_0,\Omega_0'\subset \Omega$ such that $\Omega_0$ and $\Omega_0'$ are $G$- and $H$-invariant, respectively, and such that the restrictions
\begin{align*}
i_0&:=i\restriction_{G\times Y_0} \colon G\times Y_0 \To \Omega_0\\
j_0&:=j\restriction_{H\times X_0} \colon H\times X_0 \To \Omega_0'
\end{align*}
are $G$- and $H$-equivariant, measure preserving Borel isomorphisms, respectively.  As shown above, we obtain a measure preserving near action of  $H$ on $(Y_0,\nu)$ and by  Lemma~\ref{lem:near-action-to-Borel-action} there exists a measure preserving action,$(h,y)\mapsto h.y $,  of $H$ on $Y_0$ which agrees with the near action almost everywhere. Denote by $\omega_G$ and $\omega_H$ the measurable cocycles associated with the coupling, and assume without loss of generality \cite[Theorem B.9]{zimmer-book} that both cocycles are strict.  We then get a measure preserving $H$-action on $G\times Y_{0}$ by defining $h.(g,y)=(g\omega_G(h,y)^{-1}, h.y)$ which induces a measure preserving $H$-action  on
$\Omega_{0} :=i_0(G\times Y_{0}  )$ given by $h.t=i_0(h.i_0^{-1}(t))$ that agrees with the original $H$-action on $\Omega$ almost everywhere.
Now clearly $\Omega_{0}$ is both $G$-and $H$-invariant (for the original action of $G$ and the new action of $H$) and $i_{0}$ is a measure preserving, $G\times H$-equivariant Borel isomorphism with respect to these actions. 
 Symmetrically, we obtain a measurable $G$-action on $X_0$ which pushes forward to a $G$-action on $\Omega_{0}':=j_0(H\times X_{0})$ that agrees  with the original action $G\curvearrowright \Omega$ almost everywhere and with respect to which $j_{0}$ is a $G\times H$-equivariant, measure preserving, Borel isomorphism.\\
 
  Next, replace $\eta$ with a probability measure $\zeta$ in the same measure class, and note that both the new actions $G\times H\curvearrowright \Omega_{0}$ and $G\times H\curvearrowright \Omega_{0}'$ are non-singular with respect to $\zeta$. The inclusions $\Omega_{0}\subset \Omega$ and $\Omega_{0}'\subset \Omega$ induce $G\times H$-equivariant isomorphisms at the level of measure algebras by \cite[Lemma 3.1]{ramsey}  (here $\Omega$ is considered with the original $G\times H$-action and $\Omega_{0}$ and $\Omega_{0}'$ with the ones just constructed) associated with $\zeta$ 
and in total we therefore obtain a $G\times H$-equivariant isomorphism of measure algebras $\Phi\colon B(\Omega_{0}, [\zeta])\to B(\Omega_{0}', [\zeta]).$ By Mackey's uniqueness theorem \cite[Theorem 2]{mackey-point-realization} there exist $\zeta$-conull (and thus $\eta$-conull) $G\times H$-invariant Borel subsets $\Omega_{00}\subset \Omega_{0}$ and $\Omega_{00}'\subset \Omega_{0}'$ and a $G\times H$-equivariant Borel isomorphism $\varphi \colon \Omega_{00}\to \Omega_{00}'$ which dualizes  to $\Phi$. Since $\Phi$ preserves the  measure $\eta$ the same is true for $\varphi$. Now pull $\Omega_{00}'$ back via $j_0$ to an $H$-invariant subset of $H\times X_{0}$ which is then of the form $H\times X_{00}$ for a conull Borel subset $X_{00}\subset X_{0}$ and, similarly, pull $\Omega_{00}$ back via $i_0$ to a set of the form $G\times Y_{00}$. Then replacing $(\Omega, \eta)$ with $(\Omega_{00}, \eta\restriction_{\Omega_{00}} )$ we obtain measure preserving Borel isomorphisms
\begin{align*}
i_{00}&\colon (G\times Y_{00}, \lambda_G \times \nu\restriction_{Y_{00}} ) \To (\Omega_{00},\eta\restriction_{\Omega_{00}})\\
\varphi^{-1} \circ j_{00} &\colon (H\times X_{00}, \lambda_H\times \mu\restriction_{X_{00}} ) \To  (\Omega_{00},\eta\restriction_{\Omega_{00}}),
\end{align*}
which are,  respectively,  pointwise $G$- and $H$-equivariant.  Putting $\Omega' = \Omega_{00}$, $X' = X_{00}$, $Y' = Y_{00}$, $i'=i_{00}$ and $j':=\varphi^{-1} \circ j_{00} $  we obtain a coupling with the claimed properties.
\end{proof}
\begin{rem}
A measure coupling  $(\Omega, \eta, X, \mu, Y,\nu, i,j)$ in which $i$ and $j$ are Borel isomorphisms and globally equivariant is called a \emph{strict measure coupling}, and the theorem just proven shows that there is always a strict measure coupling between measure equivalent groups.
\end{rem}

The following lemma  will be used in the proof of Proposition \ref{prop:free-ergodic-coupling} and in Section~\ref{sec:me-oe}.

\begin{lemma}
  \label{lem:invariance}
  Let $G$ be an lcsc group and $X$ be a standard Borel space and endow $G \times X$ with the  structure of a measurable $G$-space given by multiplication on the first factor.
  \begin{enumerate}
  \item[(i)] If $[\mu]$ is a probability measure class on $X$ and $[\rho]$ is a $G$-invariant class of a $\sigma$-finite measure on $G \times X$ that projects to $[\mu]$ via the right leg projection $p_X\colon G\times X \to X $ then $[\rho] = [\lambda_G \times \mu]$.
  \item[(ii)] If $\mu$ is a probability measure on $X$ and $\rho$ is a $G$-invariant $\sigma$-finite measure on $G \times X$ which is equivalent to a probability measure projecting to $\mu$, then there is a measurable function $b\colon X \rightarrow [0,\infty [$ such that $\rho = \lambda_G \times b \mu$.
  \end{enumerate}
\end{lemma}
Note that part (i) may be seen as a generalization of   \cite[Lemma 3.3]{mackey-stone-von-neumann}, and that part (ii) boils down to uniqueness of the Haar measure in the situation where  $X$ is a one-point space. 

\begin{proof}
  We start by proving (i).  
  Since $\rho$ is $\sigma$-finite it is equivalent to a probability measure $\rho_0$ and we have $\mu \sim (p_X)_*(\rho)\sim (p_X)_*(\rho_0) $, so upon replacing $\mu$ with  $(p_X)_*(\rho_0)$ we may assume that $(p_X)_*(\rho_0)=\mu$. So Theorem 2.1 in \cite{hahn-haar-measure-groupoids} applies and we find a disintegration $\rho = \int_X \rho^x \mathrm d \mu(x)$, where $\rho^x$ are Borel measures supported on $G \times \{x\}$ for each $x \in X$.  Further, the classes $[\rho^x]$ are uniquely determined $\mu$-almost everywhere and since $\rho$ is $\sigma$-finite, $\rho^x$ is $\mu$-almost everywhere $\sigma$-finite too.  We may therefore assume that $\rho^x$ is $\sigma$-finite for all $x\in X$.  \\
  Fix now a $g\in G$ and consider the measure $g_*\rho$. The family $(g_*\rho^x)_{x\in X}$ then provides a disintegration of $g_*\rho$, and since $g_*\rho\sim \rho$ by assumption, we conclude from the uniqueness part of  \cite[Theorem 2.1]{hahn-haar-measure-groupoids} that $g_*\rho^x\sim \rho^x$ for almost all $x\in X$. Thus, by Fubini's theorem, we conclude that for almost all $x\in X$, $g_*\rho^x\sim \rho^x$ for almost all $g\in G$.  Fixing an $x\in X$ such that $g_*\rho^{x}\sim \rho^{x}$ for almost all $g\in G$, one observes that the set $\{g\in G \mid g_*\rho^x\sim \rho^x \}$ is closed under the multiplication in $G$, and hence has to be equal to $G$ by \cite[Proposition B.1]{zimmer-book}. This shows that, $\mu$-almost everywhere, the measure class $[\rho^x]$ is $G$-invariant.  By Mackey's result \cite[Lemma 3.3]{mackey-stone-von-neumann}, this implies that $[\rho^x] = [\lambda_G\times \delta_x]$, where $\delta_x$ denotes the Dirac mass at $x$.  We therefore obtain that
\begin{equation*}
  \rho
  \sim
  \int_X \rho^x \mathrm d \mu(x)
  \sim 
  \int_X \lambda_G \times \delta_x \mathrm d \mu(x)
  =
  \lambda_G \times \mu
  \text{,}
\end{equation*}
which finishes the proof of (i). \\

We next prove (ii).  By Theorem 2.1 in \cite{hahn-haar-measure-groupoids} there exists a disintegration $\rho = \int_X \rho^x \mathrm d \mu(x)$ where, for all $x \in X$, each $\rho^x$ is a Borel measure on $G \times X$ whose support lies in $G \times \{x\}$ and the map  $x\mapsto \rho^x$ is unique up to measure zero.  For  notational convenience, whenever $x\in X$ we will identify $G$ with $G\times \{x\}$ in the sequel.
For fixed $g \in G$, the equality
\begin{gather*}
  \int_X \rho^x \mathrm d \mu(x) = \rho = g_* \rho = \int_X g_*\rho^x \mathrm d \mu(x)
\end{gather*}
shows that $g_* \rho^x = \rho^x$ for $\mu$-almost all $x \in X$.  By Fubini, this implies that for $\mu$-almost all $x \in X$ the subgroup $\{g \in G \mid g_* \rho^x = \rho^x\} \leq G$ is co-negligible and must therefore equal $G$ by \cite[Proposition B.1]{zimmer-book}.  We therefore obtain that $\rho^x$ is $G$-invariant for $\mu$-almost all $x \in X$.  Replacing $\rho^x$ by $\lambda_G$ on a $\mu$-negligible subset, we may assume that $\rho^x$ is $G$-invariant for all $x \in X$.  Hence, uniqueness of the Haar measure implies that for all $x \in X$ there is some $b(x) \in [0, \infty[$ such that $\rho^x = b(x) \lambda_G$.  Choosing a measurable subset $E \subset G$ satisfying $\lambda_G(E) = 1$, we find that that $b(x) = \rho^x(E \times X)$ for all $x \in X$, which proves measurability of $b: X \to [0,\infty[$.  We infer that $\rho = \lambda_G \times b \mu$.  This finishes the proof of the lemma.\end{proof}

We end this section with a result showing that in addition to strictness, one can also obtain freeness and ergodicity of measure couplings.  In order to clarify the measure theoretical subtleties, we recall the following definitions.
\begin{defi}
  \label{def:free-ergodic}
  Let $G$ be an lcsc group and $(X,\mu)$ a non-singular $G$-space.
  \begin{itemize}
  \item We say that $G \curvearrowright X$ is \emph{essentially free} if the set of all elements in $X$ whose stabiliser is non-trivial is conull.
  \item We say that $G \curvearrowright X$ is \emph{ergodic} if every measurable subset $A \subset X$, for which the set $A \Delta g A \subset X$ is a null set  for all $g \in G$, is either null or conull in $X$.
  \end{itemize}
\end{defi}
\begin{rem}\label{rem:strong-ergodicity}
Note that when $G$ acts non-singularly on a standard Borel probability space $(X,\mu)$, then ergodicity of the action is equivalent to the formally weaker statement:  any  $G$-invariant Borel set $A\subset X$ is either null or conull (cf.~\cite[Proposition 7.7]{grigorchuk-de-la-harpe}). In particular, when given a standard Borel $G$-space with  a $\sigma$-finite $G$-invariant Borel measure $\eta$, we may find a probability measure $\mu$ on $X$ such that $[\eta]=[\mu]$ and for which the action is therefore non-singular. Hence in this situation there is no difference between the two notions of ergodicity.

\end{rem}

\begin{prop}
  \label{prop:free-ergodic-coupling}
Let $G$ and $H$ be measure equivalent, unimodular, lcsc groups. Then there exists a free, ergodic and strict measure equivalence coupling between $G$ and $H$.
\end{prop}
\begin{proof}
We split the proof into {two} statements.
\begin{itemize}
\item[(i)] There exists a strict measure equivalence coupling between $G$ and $H$ with genuinely free $G\times H$-action. 
\item[(ii)] The measures associated with the coupling in (i) can be replaced by ergodic ones.
\end{itemize}

We first prove (i). Assume, as we may by Theorem \ref{thm:strict-coupling}, that $(\Omega,X,Y,i,j)$ is a strict measure equivalence coupling and let $(Z,\zeta)$ be a standard Borel probability space upon which $G\times H$ acts freely and measure preservingly, the existence of which is guaranteed by \cite[Proposition 1.2]{adams-elliot-giordano} combined with \cite[Lemma 10]{stefaan-sven-niels}.  Now consider $(\Omega',\eta'):=(\Omega\times Z, \eta\times \zeta)$ with the diagonal $G\times H$-action.  The Borel isomorphism $i^{-1}\times \id \colon \Omega'\to G\times Y\times Z$ intertwines the $G$-action on $\Omega'$ with the action on $G \times Y \times Z$ given by $g_0.(g,z,y) \colon=(g_0g,y,g_0.z)$. 
Further, the Borel isomorphism $\alpha\colon G\times Y\times  Z\to G\times Y\times Z$ given by $\alpha(g,y,z):=(g,y,g^{-1}.z)$ preserves $\lambda_G\times \nu \times \zeta$ and intertwines the $G$-action just described with the action given by multiplication on the first leg. Thus, setting  $(Y',\nu'):=(Y\times Z,\nu\times \zeta)$, the map $(i\times \id)\circ \alpha^{-1}\colon G\times Y' \to \Omega'$ is a Borel isomorphism intertwining the $G$-action on the first leg with the $G$-action on $\Omega'$.  By symmetry, this shows that $\Omega'$ is a strict measure equivalence coupling and, by construction, the $G\times H$-action on $\Omega'$ is free.

To prove (ii), assume that $(\Omega,\eta, X,\mu,Y,\nu, i,j)$ is  a strict, free measure equivalence coupling and hence that $i$ and $j$ induce measure preserving actions $G\curvearrowright (X,\mu)$ and $H\curvearrowright (Y,\nu)$.
  By the Ergodic Decomposition Theorem \cite[Theorem 7.8]{grigorchuk-de-la-harpe} there exists a standard Borel probability space $(Z,\zeta)$ and a family of $G$-invariant, ergodic probability measures $(\mu_z)_{z\in Z}$ such that the map $z\mapsto \mu_z(A)$ is measurable for every measurable subset $A\subset X$ and
\[
\mu(A)=\int_Z \mu_z(A) \mathrm{d}\zeta(z).
\]
Since $\mu_z$ is $G$-invariant and $H$ is unimodular, the pushforward $\eta_z:=j_*(\lambda_H\times \mu_z)$ is  $G\times H$-invariant  and hence the same is true for $\rho_z:=(i^{-1})_*\eta_z$ on $G\times Y$ when the latter is considered with the $G\times H$-action induced by $i^{-1}$. As $\eta=j_*(\lambda_H\times \mu)$ we get that $\eta(B)=\int_{Z}\eta_z(B)d\zeta(z)$ for every measurable set $B\subset \Omega$ and since  $\eta$ is $\sigma$-finite,  $\eta_z$ is $\sigma$-finite for almost all $z\in Z$ and hence the same is true for $\rho_z$.  
Denote by $Z_0$ the  conull Borel subset consisting of points $z$ for which $\eta_z$ is $\sigma$-finite
 and pick a $z \in Z_0$; we now show that $\eta_z$ is ergodic. To this end, note first that Remark \ref{rem:strong-ergodicity} applies so it suffices to check ergodicity only on genuinely $G \times H$-invariant measurable subsets $B\subset \Omega$. For such a set $B$, since the coupling is strict  we obtain that $j^{-1}(B)$ is $H$-invariant and hence of the form $H\times B_0$ for a Borel subset $B_0\subset X$. By $G$-invariance of $B$ we conclude that $B_0$ is invariant for the induced action $G \curvearrowright X$ and hence either null or conull with respect to $\mu_z$; thus $B$ is  either null or conull with respect to $\eta_z$. From this we conclude that $\rho_z$ is ergodic for the $G\times H$-action induced via $i$.  Since $\rho_z$ is $\sigma$-finite it is equivalent to a probability measure $\rho_z'$ and applying Lemma \ref{lem:invariance} to the probability measure $\nu_z:=(\pi_Y)_*\rho_z'$ we obtain a measurable function $b_z\colon Y\to [0,\infty[$ such that $\rho_z=\lambda_G\times b_z\nu_z$. Now note that since $\rho_z$ is ergodic for the $G\times H$-action  $\nu_z':=b_z\nu_z$ is ergodic for the $H$-action; to see this let $B\subset Y$ be an $H$-invariant Borel set and note that $G\times B$ is $G\times H$-invariant and hence either null or conull with respect to $\rho_z=\lambda_G\times \nu_z'$.  Moreover, for any Borel set $U\subset G$ we have 
\[
\lambda_G(U)\nu(Y)=(\lambda_G\times \nu)(U\times Y)= \int_{Z}\rho_z(U\times Y)\mathrm d\zeta(z)= \int_{Z}\lambda_G(U)\nu_z'(Y)\mathrm d\zeta(z),
\]
so $\nu_z'$ is finite $\zeta$-almost surely. Hence for any choice of $z \in Z$ such that $\nu_z'$ is finite, the measure equivalence coupling $(\Omega, \eta_z,X,\mu_z, Y,\nu_z', i,j)$ is ergodic.
\end{proof}

\section{Measure equivalence and orbit equivalence}
\label{sec:me-oe}

We now introduce orbit equivalence in the setting of non-discrete lcsc groups, as treated for instance in \cite[Definition 1.12]{carderi-maitre} for ergodic actions. We specify our definition as follows.
\begin{defi}
\label{def:soe}
Let $G$ and $H$ be lcsc groups and let $(X,\mu)$ be an essentially free,
non-singular, standard Borel probability $G$-space and $(Y,\nu)$ be an essentially free, non-singular,  standard Borel probability $H$-space.  

\begin{itemize}
\item The two actions are said to be \emph{orbit equivalent} if there exist a Borel map $\Delta\colon X\to Y$,  conull Borel subsets $X_0\subset X$ and $Y_0\subset Y$ such that $\Delta_*\mu \sim \nu$ and  $\Delta$ restricts to a Borel isomorphism $\Delta_0\colon X_0 \to Y_0$ 
with the property that
\begin{align}\label{eq:OE-defi}
\Delta_0(G.x \cap X_0)=H.\Delta_0(x)\cap Y_0 \text{ for all } x\in X_0.
\end{align}
\item 
If the actions are actually measure preserving then they are said to be \emph{orbit equivalent} if they are so as non-singular  actions and if the map $\Delta$ can be chosen such that $\Delta_*({\mu})={\nu} $. Moreover, the actions are said to be \emph{stably orbit equivalent} if there exist non-negligible Borel subsets $A\subset X$ and $B\subset Y$ such that $G.A\subset X$ is conull, $H.B\subset Y$ is conull and there exists a Borel isomorphism $\Delta \colon A\to B$  such that $\Delta_*(\mu(A)^{-1}\mu\restriction_{A} )=\nu(B)^{-1}\nu\restriction_{B}$ and such that
\begin{align}\label{eq:SOE-defi}
\Delta(G.a \cap A)=H.\Delta(a)\cap B \text{ for all } a\in A.
\end{align}
\end{itemize}
\end{defi}

\begin{rem}
Note that the conull subsets in the definition of orbit equivalence are required to be Borel, but reasoning as in Remark \ref{rem:borel-vs-non-borel} one sees that this is equivalent to the corresponding definition without the Borel requirement (and with measures replaced by completed measures). Note also that if the two actions are orbit equivalent so are their restrictions to any choice of conull, invariant subsets in $X$ and $Y$, respectively.
\end{rem}

\begin{rem}\label{rem:soe}
Note that if $G$ and $H$ are non-discrete lcsc groups that admit stably orbit equivalent, ergodic actions then the original actions are actually already orbit equivalent \cite[Lemma~1.20]{carderi-maitre}. Thus, the notion of stable orbit equivalence is really only of interest for discrete groups.
\end{rem}

\begin{lem}
  \label{lem:cocycles}
 If $G\curvearrowright (X,\mu)$ and $H\curvearrowright (Y,\nu)$ are non-singular, essentially free actions of lcsc groups $G$ and $H$ on standard Borel probability spaces and if $\Delta\colon X \rightarrow Y$ is an orbit equivalence between the two actions, then there exist  measurable cocycles $c\colon G\times X\to H$ and $d\colon H\times Y\to G$ with the properties that:
 \begin{itemize}
\item [(i)] for all $g\in G$  the relation $\Delta(g.x)=c(g,x).\Delta(x)$ holds for almost all $x\in X$;
\item[(ii)]  for all $h\in H$ the relation $\Delta^{-1}(h.y)=d(h,y).\Delta^{-1}(y)$ holds for almost all $y\in Y$.
\end{itemize}  
Here $\Delta^{-1}$ denotes any Borel extension to $Y$ of $\Delta_0^{-1}\colon Y_0\to X_0$.   
\end{lem}
\begin{proof}
The set of points with trivial stabilizer is invariant, conull and Borel  \cite[Lemma 10]{stefaan-sven-niels}, so we may assume that both actions are genuinely free.
Fix $X_0 $ and $Y_0$ as in the definition and a $g\in G$. Then $g^{-1}X_0\cap X_0:=X_g$ is conull in $X$ and for $x\in X_g$ we have $\Delta(g.x)\in H.\Delta(x)\cap Y_0$ so there exists a unique (by freeness) element $c(g,x)\in H$ such that $\Delta(g.x)=c(g,x).\Delta(x)$. This defines a map
\[
c\colon \{(g,x)\in G\times X \mid x\in X_g\} \to H,
\]
and the domain of $c$ is conull in $G\times X$ by Fubini, since each $X_g$ is conull in $X$. Extend $c$ to all of $G\times X$ by mapping elements in the complement of $\{(g,x)\in G\times X \mid x\in X_g\} $ to $e\in H$. Then $c$ is a measurable cocycle, because for all $g_1,g_2\in G$ and all $x$ in the conull subset $X_{g_1g_2}\cap g_2^{-1}.X_{g_1} \cap X_{g_2}$ we have
\[
c(g_1g_2,x).\Delta(x)=\Delta(g_1g_2.x)=\Delta(g_1.(g_2.x))=c(g_1,g_2.x).\Delta(g_2.x)=c(g_1,g_2.x)c(g_2,x).\Delta(x),
\]
and hence $c(g_1g_2,x)=c(g_1,g_2.x)c(g_2,x)$ by freeness. Moreover, by construction $\Delta(g.x)=c(g,x).\Delta(x)$ for all $x$ in the conull subset $X_g$. The existence of the cocycle $d$ follows by symmetry.
\end{proof}

The cocycles constructed in Lemma \ref{lem:cocycles} are compatible with the natural measures on $G \times X$ and $H \times X$ as the next lemma describes.  It will be used in the proof of Lemma \ref{oe-implies-me}
\begin{lem}
  \label{lem:cocycle-measure-preserving}
  Let $G$ and $H$ be lcsc groups and let $(X, \mu)$ be an essentially free, ergodic, pmp action on standard Borel probability $G$-space and $(Y, \nu)$ an essentially free, ergodic, pmp, standard Borel probability $H$-space.
 If $\Delta\colon X \rightarrow Y$ is an orbit equivalence between the two actions and $c\colon G \times X \rightarrow H$ the associated cocycle, then $\Phi \colon G \times X \to H \times Y,  (g,x) \mapsto (c(g,x), \Delta(x)),$ satisfies $\Phi_*[\lambda_G \times \mu] = [\lambda_H \times \nu]$.
\end{lem}
\begin{proof}
  The push-forward $\Phi_*[\lambda_G \times \mu]$ projects onto the class  $\Delta_*[\mu] = [\nu]$.  Further, it satisfies the condition of Definition 2.3 in \cite{hahn-haar-measure-groupoids}, saying that $(H \times Y, \Phi_*[\lambda_G \times \mu])$ is a measure groupoid.  So the uniqueness statement for invariant measures classes on groupoids formulated as Proposition 3.4 of \cite{feldman-hahn-moore} applies and shows that $\Phi_*[\lambda_G \times \mu] = [\lambda_H \times \nu]$.
\end{proof}


\subsection{Cross section equivalence relations}
\label{sec:cross-sections}

In this section we briefly recall the notion of a cross section for an action of a locally compact group, which is originally due to Forrest \cite{Forrest} and more recently treated in \cite{KPV} and \cite{carderi-maitre}.  Let $G$ be an lcsc  group  and let $(X,\mu)$ be a standard Borel probability space endowed with a non-singular, essentially free action $\theta\colon G \curvearrowright (X,\mu)$. It is well known that it is, in general, impossible to choose a Borel subset of $X$ meeting every orbit exactly once, but by \cite[Proposition 2.10]{Forrest} one may find a Borel subset $X_0\subset X$ and an open neighbourhood of the identity $U\subset G$ such that 
\begin{itemize}
\item[(i)] the restricted action map $\theta\restriction \colon U\times X_0 \to X$ is injective, and
\item[(ii)] the subset $G.X_0$ is Borel and conull in $X$.
\end{itemize}
A subset $X_0\subset X$ with the above mentioned properties is called a \emph{cross section} of the action $G\curvearrowright (X,\mu)$. Note that since $G$ is assumed second countable, if $\theta\restriction \colon U\times X_0 \to X$ is injective then $\theta\restriction\colon G\times X_0\to X$ is countable-to-one and hence
the set 
\[
Z:=\{(x,x_0)\in X\times X_0\mid x\in G.x_0 \} 
\]
is Borel and the projection $\pi_l\colon Z\to X$ is countable-to-one. One may therefore  define a $\sigma$-finite Borel measure $\rho$ on $Z$ by setting
\[
\rho(E):=\int_{X}|\pi_l^{-1}(x)\cap E| \, d\mu(x),
\] 
where $|\cdot |$ here, and in what follows, denotes the cardinality of the set in question.
In the situation where $G$ is unimodular and the action is assumed measure preserving one has the following.

\begin{prop}[{cf.~\cite[Proposition 4.3]{KPV}}]\label{prop:cross-section}
Let $G$ be an unimodular lcsc  group and $(X,\mu)$ be a standard Borel probability space endowed with an essentially free, pmp action $G\overset{\theta}{\curvearrowright} (X,\mu)$. Denote by $X_0\subset X$ a cross section and fix a Haar measure $\lambda_G$ on $G $. Then the following hold.
\begin{enumerate}
\item The set $\R_{X_0}:=\{(x,x')\in X_0\times X_0 \mid x\in G.x'\}$ is a Borel equivalence relation with countable orbits.
\item 

There exist a unique Borel probability measure $\mu_0$ on $X_0$ and a number $\covol(X_0)\in ]0,\infty[$ such that the map $\Psi\colon G\times X_0\to Z$ given by $\Psi(g,x_0):=(g.x_0,x_0)$ satisfies
\begin{align*}
\Psi_*(\lambda_G\times \mu_0)=\covol(X_0)\rho,
\end{align*}
 where $Z$ and $\rho$ are as defined above.
Thus,  whenever $U\subset G$ is an open identity neighbourhood with $\theta\restriction\colon U\times X_0 \to X$ injective then 
\[
(\theta\restriction)_* (\lambda_G\restriction_U \times \mu_0)=\covol(X_0)\mu\restriction_{U.X_0}.
\]
Moreover, the measure $\mu_0$ is invariant under the equivalence relation $\R_{X_0}$.
\item  The action $G\curvearrowright (X,\mu)$ is ergodic if and only if $(\R_{X_0},\mu_0)$ is ergodic and in this case the equivalence relation associated with another choice of cross section is stably orbit equivalent to $\R_{X_0}$.
\item The group $G$ is  amenable if and only if $(\R_{X_0},\mu_0)$ is amenable.
\end{enumerate}
\end{prop}

The equivalence  relation $\R_{X_0}$ is referred to as a \emph{cross section equivalence relation} for the action $G\curvearrowright (X,\mu)$.
For background material about countable equivalence relations and their properties we refer to \cite{FM1} and \cite{Connes-Feldman-Weiss}. At this point we just single out our definition of stable orbit equivalence, since we were unable to find a suitable reference for this in the non-ergodic case (see eg.~\cite{Furman-OE-rigidity} for the ergodic case).
\begin{defi}\label{def:OE-for-eq-rel}
Let $\R$ and $\S$ be  countable Borel measure preserving equivalence relations on  standard Borel probability spaces $(X,\mu)$ and $(Y,\nu)$. Then $\R$ and $\S$ are said to be \emph{stably orbit equivalent} if there exist non-negligible Borel subsets $A\subset X$ and $B\subset Y$ and a Borel isomorphism $\Delta\colon A\to B$ such that:
\begin{itemize}
\item[(i)] the $\R$-saturation of  $A$, i.e.~the set $\{x\in X\mid \exists a\in A: (x,a)\in \R\}$, is conull in $X$ and the $\S$-saturation of $B$ is conull in $Y$;
\item[(ii)] for all $a,a'\in A$: $(a,a')\in \R$ if and only if $(\Delta(a), \Delta(a'))\in \S$;
\item[(iii)] we have $\Delta_*(\mu(A)^{-1}\mu\restriction_A)=\nu(B)^{-1}\nu\restriction_B$.
\end{itemize}
\end{defi}

\subsection{Equivalence of measure equivalence and stable orbit equivalence}
\label{sec:equivalence-me-soe}
The aim in the section is to provide a proof of Theorem \ref{thm:me-oe-soe}; more precisely we prove the following.

\begin{thm}\label{thm:me-and-soe-equivalence}
For unimodular, lcsc groups $G$ and $H$ the following are equivalent.
\begin{itemize}
\item[(i)] $G$ and $H$ are measure equivalent.
\item[(ii)] $G \times \mathrm S^1$ and $H \times \mathrm S^1$ admit  orbit equivalent, essentially free, ergodic, pmp actions on standard Borel probability spaces.
\item[(iii)] $G$ and $H$ admit essentially free, ergodic, pmp actions on standard Borel probability spaces for which the cross section equivalence relations associated with some (equivalently any) choice of cross sections are stably orbit equivalent.
\end{itemize}
If both groups $G$ and $H$ 
are discrete,  (ii) can be replaced by
\begin{itemize}
\item[(ii)'] $G$ and $H$ admit stably orbit equivalent, essentially free, ergodic, pmp actions on standard Borel probability spaces.
\end{itemize}
If both $G$ and $H$ are non-discrete, (ii) can be replaced by
\begin{itemize}
\item[(ii)''] $G$ and $H$ admit orbit equivalent, essentially free, ergodic, pmp actions on standard Borel probability spaces.
\end{itemize}

\end{thm}
Here, and in what follows, $ \mathrm S^1$ denotes the circle group.
We split up the proof of Theorem \ref{thm:me-and-soe-equivalence} into several lemmas, some of which are, in the interest of future reference, stated in slightly more generality than needed for  Theorem \ref{thm:me-and-soe-equivalence}.

\begin{lemma}
\label{lem:SOE-cross-eq-implies-OE-actions}
If $G$ and $H$ are non-discrete, unimodular, lcsc groups with essentially free, ergodic, pmp actions $G\curvearrowright (X,\mu)$ and $H\curvearrowright (Y,\nu)$ on standard Borel probability spaces  and if $X_0\subset X$ and $Y_0\subset Y$ are cross sections for the two actions whose associated cross section equivalence relations are stably orbit equivalent then the original actions are orbit equivalent.
\end{lemma}

\begin{proof}
  Let $U\subset G$ and $V\subset H$ be  precompact, open neighbourhoods of the identities in $G$ and $H$ such that the restricted actions maps $\alpha\colon U\times X_0 \to X$ and $\beta\colon V\times Y_0 \to Y$ are injective and hence Borel isomorphisms onto their image by Lemma \ref{lem:injective-borel}.  Since $G$ and $H$ are assumed non-discrete, we may choose an isomorphism of standard Borel probability spaces $f\colon (U, \lambda_G(U)^{-1}\lambda_G\restriction_U) \to (V, \lambda_H(V)^{-1}\lambda_H\restriction_H )$ (see \cite[Theorem (17.41)]{kechris-book}). By assumption there exist non-negligible Borel subsets $A\subset X_0 $ and $B\subset Y_0$ and a Borel isomorphism $\Delta\colon A\to B$ preserving orbits and the restricted (normalized) measures.  Then the map $\tilde{\Delta}\colon U.A \to V.B$ given by
\[
 U.A \overset{\alpha^{-1}}{\To} U\times A \overset{f\times \Delta}{\To} V\times B \overset{\beta}{\To} V.B
\]
is an isomorphism of probability measure spaces by Proposition \ref{prop:cross-section} (2).  Further, it preserves orbits, since $\Delta$ is an orbit equivalence and for $u_1,u_2\in U$ and $a_1,a_2\in A$ we have $u_1a_1 \sim u_2a_2 $ iff $a_1\sim a_2$.  It follows that $\tilde{\Delta}$ is an orbit equivalence, which shows that the original actions $G\curvearrowright (X,\mu)$ and $H\curvearrowright (Y,\nu)$ are stably orbit equivalent and since both $G$ and $H$ are assumed non-discrete the two actions are actually orbit equivalent by \cite[Lemma 1.19]{carderi-maitre}.
\end{proof}

\begin{lem}
\label{lem:ME-implies-SOE-cross}
If $G$ and $H$ are measure equivalent,  unimodular, lcsc groups then for any strict, essentially free, measure equivalence coupling $(\Omega,\eta,X,\mu,Y,\nu, i,j)$ the associated  pmp actions $G\curvearrowright (X,\mu)$ and $H\curvearrowright (Y,\nu)$ are essentially free and for any choice of cross sections $X_0\subset X$ and $Y_0\subset Y$ the associated cross section equivalence relations  are stably orbit equivalent.
\end{lem}
Note that measure equivalence couplings with the properties prescribed in Lemma \ref{lem:ME-implies-SOE-cross} always exist by Proposition \ref{prop:free-ergodic-coupling}. The statement in Lemma \ref{lem:ME-implies-SOE-cross} is implicit in the proofs of Proposition 4.3 and Corollary 4.6 in  \cite{KPV}, but for the convenience of the reader we provide a full proof below.
\begin{proof}
Let $(\Omega, \eta, X,\mu, Y,\nu, i,j)$ be an essentially free, strict measure equivalence coupling between $G$ and $H$,
and note that we may assume that it is genuinely free by Proposition \ref{prop:free-ergodic-coupling}.
 Since $G\times H\curvearrowright \Omega$ is free so are both the induced actions $G {\curvearrowright} (X,\mu)$ and $H{\curvearrowright} (Y,\nu)$
  and we can therefore find cross sections $X_0\subset X$ and $Y_0\subset Y$ with associated probability measures $\mu_0$ and $\nu_0$  satisfying the conditions in Proposition \ref{prop:cross-section} (2). 
Denote by $\R_{X_0}$ and $\R_{Y_0}$ the associated cross section equivalence relations; we now need to show that these are stably orbit equivalent.
Note that $X_0':=j(\{e_H\} \times X_0)$ and $Y_0':=i(\{e_G\} \times Y_0)$ are both cross sections\footnote{Strictly speaking, we have only defined cross sections for non-singular actions on probability spaces, but since $\eta$ is $\sigma$-finite we may replace it with an equivalent probability measure with respect to which the action is then non-singular.} for the action $G\times H\curvearrowright \Omega$ 
and we will now consider them with the probability measures $\mu_0':=j_*(\delta_{e_H}\times \mu_0)$ and $\nu_0':=i_*(\delta_{e_G}\times \nu_0)$. 
Note also that the map $x_0\mapsto j(e_H,x_0)$ is an orbit equivalence between $\R_{X_0}$ and the restriction to $X_0'$ of the orbit equivalence relation $\R_{G\times H}$ of $G\times H\curvearrowright \Omega$. It therefore suffices to show that $\R_{G\times H}\restriction_{X_0'}$ is stably orbit equivalent with $\R_{G\times H}\restriction_{Y_0'}$. For notational convenience we put $K:=G\times H$ and define
\begin{align*}
\S&:= \{(x_0,y_0)\in X_0'\times Y_0' \mid y_0\in K.x_0 \},\\
\mathcal{Z}&:=\{(t,x_0,y_0)\in \Omega \times X_0' \times Y_0' \mid K.t=K.x_0=K.y_0\},\\
\mathcal{Z}_{X_0'} &:=\{(t,x_0)\in \Omega \times X_0' \mid t\in K.x_0\},\\
\mathcal{Z}_{Y_0'} &:=\{(t,y_0)\in \Omega \times Y_0' \mid t\in K.y_0\},\\
Z_{X_0} &:= \{(x,x_0)\in X\times X_0 \mid  x\in G.x_0\},\\
Z_{Y_0} &:= \{(y,y_0)\in Y\times Y_0 \mid  y\in H.y_0\}.\\
\end{align*}
Each of these sets is Borel in the product space in question and we now endow them with $\sigma$-finite measures. Note that the projection maps $\pi_l\colon \S\to X_0'$ and $\pi_r\colon \S\to Y_0' $ are both countable-to-one and we can therefore define two measures on $\S$ by integrating the counting measure against the measures on $X_0'$ and $Y_0'$, respectively. In more detail, for a Borel subset $E\subset \S$ we define
\begin{align*}
\gamma_l(E)&:=\int_{X_0'} |\pi_l^{-1}(x_0) \cap E| \mathrm{d}\mu_0'(x_0),\\
\gamma_r(E)&:=\int_{Y_0'} |\pi_r^{-1}(y_0) \cap E| \mathrm{d}\nu_0'(y_0).
\end{align*}
 Similarly, the left leg projection maps from $\Z, \Z_{X_0'}$ and $\Z_{Y_0'}$ onto $\Omega$ are countable-to-one and we obtain $\sigma$-finite measures $\eta_{\Z}, \eta_{\Z_{X_0'}}$ and $\eta_{\Z_{Y_0'}}$ by integrating the counting measure along it against the original measure $\eta$ on $\Omega$.

\begin{claim}\label{pf-formula-1}
The Borel isomorphisms $\alpha\colon K\times X_0' \to \Z_{X_0'}$ and $\beta\colon K\times Y_0'\to \Z_{Y_0'}$ given by $\alpha(k,x_{0}):=(k.x_0,x_0)$ and $\beta(k,y_{0})=(k.y_0,y_0)$ satisfy $\alpha_*(\lambda_K\times \mu_0')= \covol(X_0)\eta_{\Z_{X_0'}}$ and  $\beta_*(\lambda_K\times \nu_0')=\covol(Y_0)\eta_{\Z_{Y_0'}}$.
\end{claim}
\begin{proof}[Proof of Claim \ref{pf-formula-1}]
For any Borel function $f\colon \Z_{X_0'}\to [0,\infty[$ we get, using the unimodularity of $H$ and the measure $\rho$ associated with $X_0$ via Proposition \ref{prop:cross-section}, that
\begin{align*}
  \int_{\Z_{X_0'}} f  \mathrm{d} \alpha_*(\lambda_K\times \mu_0') &= \int_{G}\int_H\int_{X_0} f((g,h).j(e_H,x_0), j(e_H,x_0))\mathrm{d}\lambda_G(g)\mathrm{d}\lambda_H(h)\mathrm{d}\mu_0(x_0)\\
&= \int_{G}\int_{X_0}\int_{H} f(j(h\omega_H(g,x_0)^{-1},g.x_0), j(e_H,x_0))\mathrm{d}\lambda_H(h)\mathrm{d}\mu_0(x_0) \mathrm{d}\lambda_G(g)\\
&= \int_{G}\int_{X_0}\int_{H} f(j(h,g.x_0), j(e_H,x_0))d\lambda_H(h)\mathrm{d}\mu_0(x_0) \mathrm{d}\lambda_G(g)\\
&= \int_H \int_{G}\int_{X_0} f\circ (j\times j )(h,g.x_0, e_H,x_0) \mathrm{d}\mu_0(x_0)\mathrm{d}\lambda_G(g)\mathrm{d}\lambda_H(h)\\
&= \covol(X_0) \int_H \int_{Z_{X_0}} f\circ (j\times j ) (h,x,e_H, x_0)\mathrm{d}\rho(x,x_0)\mathrm{d}\lambda_H(h)\\
&= \covol(X_0) \int_H \int_X \sum_{\substack{x_0\in X_0\\ x\in G.x_0}} f (j(h,x),j(e_H,x_0))\mathrm{d}\mu(x)\mathrm{d}\lambda_H(h)\\
&=\covol(X_0)\int_{H} \int_{X} \sum_{\substack{x_0\in X_0\\ j(h,x)\in j(H\times G.x_0)}} \hspace{-0.3cm} f(j(h,x_0),j(e_H,x_0))\mathrm{d}\mu(x)\mathrm{d}\lambda_H(h)\\
&=\covol(X_0) \int_\Omega \sum_{\substack{x_0'\in X_0'\\ t\in K.x_0'}} f(t,x_0')  \mathrm{d}\eta(t)\\
&=\covol(X_0)\int_{\Z_{X_0}'} f\mathrm{d}\eta_{\Z_{X_0'}}.
\end{align*}
This proves that $\alpha_*(\lambda_K\times \mu_0')= \covol(X_0)\eta_{\Z_{X_0'}}$ and the formula  $\beta_*(\lambda_K\times \nu_0')=\covol(Y_0)\eta_{\Z_{Y_0'}}$ follows by a similar argument.
\end{proof}

\begin{claim}\label{pf-formula-2}
The Borel isomorphisms $\Phi_1\colon K\times \S \to \Z$ and $\Phi_2\colon K\times \S \to \Z$ given by $\Phi_1(k,x_0,y_0)=(k.x_0,x_0,y_0)$ and $\Phi_2(k,x_0,y_0)=(k.y_0,x_0,y_0)$ satisfy ${\Phi_1}_*(\lambda_K\times \gamma_l)=\covol(X_0)\eta_{\Z}$ and ${\Phi_2}_*(\lambda_K\times \gamma_r)=\covol(Y_0)\eta_{\Z}$.
\end{claim}
\begin{proof}[Proof of Claim \ref{pf-formula-2}]
 For any Borel function $f\colon \Z\to [0,\infty[$ we have 

\begin{align*}
\int_{\Z} f{\mathrm d}{\Phi_1}_*(\lambda_K\times \gamma_l) &= \int_{K}\int_{\S} f(k.x_0,x_0,y_0) \mathrm{d}\gamma_l(x_0,y_0)\mathrm{d}\lambda_K(k)\\
&= \int_K \int_{X_0'} \sum_{y_0\in K.x_0\cap Y_0'}f(k.x_0,x_0,y_0) \mathrm{d}\mu_{0}'(x_0)\mathrm{d}\lambda_K(k)\\
&= \covol(X_0) \int_{\Z_{X_0'}} \sum_{y_0\in K.x_0\cap Y_0'}f(t,x_0,y_0) \mathrm{d}\eta_{\Z_{X_0'}}(t,x_0) \tag{by Claim \ref{pf-formula-1} }\\
&= \covol(X_0) \int_{\Omega} \sum_{x_0\in K.t\cap X_0'} \sum_{y_0\in K.x_0\cap Y_0'}f(t,x_0,y_0) \mathrm{d}\eta(t)\\
&= \covol(X_0) \int_{\Z}f \mathrm{d}\eta_{\Z}
\end{align*}
This proves the first formula claimed and the other one follows similarly.
\end{proof}
Since $K\curvearrowright \Omega$ is assumed free, for every $(x_0,y_0)\in \S$ there exists a unique $\Lambda(x_0,y_0)\in K$ such that $\Lambda(x_0,y_0)y_0=x_0$ and the map $\Lambda\colon \S\to K$ defined this way is Borel:  Indeed $X_0' \times Y_0'$ and $\mathcal R_{G \times H}$ are Borel subsets of $\Omega \times \Omega$.  So if $I$ denotes the inverse image of $X_0' \times Y_0' \cap \mathcal R_{G \times H}$ under the action map $\delta\colon K \times \Omega \to \Omega \times \Omega$, then Lemma \ref{lem:injective-borel} shows that $\delta|_I$ is a Borel isomorphism onto its image.  Thus $\Lambda = p_K \circ (\delta_I)^{-1}$ is a Borel map, where $p_K\colon K \times \Omega \to K$  denotes the first coordinate projection.  Consider now the Borel map $\zeta\colon K\times \S \to K\times \S$ given by $\zeta(k,x_0,y_0)=(k\Lambda(x_0,y_0), x_0, y_0)$ and note that since $K$ is unimodular $\zeta_*(\lambda_K\times \gamma_l)=\lambda_K\times \gamma_l$. Moreover, a direct computation shows that $\Phi_1=\Phi_2\circ \zeta$ and hence we conclude that
\begin{align*}
\frac{1}{\covol(Y_0)}{\Phi_2}_*(\lambda_K\times \gamma_r) &= \frac{1}{\covol(X_0)}{\Phi_1}_*(\lambda_K\times \gamma_l) \tag{Claim \ref{pf-formula-2}}\\
&= \frac{1}{\covol(X_0)}(\Phi_2\circ \zeta)_*(\lambda_K\times \gamma_l)\\
&= \frac{1}{\covol(X_0)}{\Phi_2}_*(\lambda_K\times \gamma_l),
\end{align*}
and since $\Phi_2$ is a Borel isomorphism this implies that $\covol(X_0)(\lambda_K\times \gamma_r)=\covol(Y_0)(\lambda_K\times \gamma_l)$ which, in turn, yields 
\begin{align}\label{left-right-formula}
\covol(X_0)\gamma_r=\covol(Y_0)\gamma_l.
\end{align}
With the formula \eqref{left-right-formula} at our disposal we can now finish the proof. Consider any $K$-invariant, $\eta$-non-negligible measurable subset $\Omega_1\subset \Omega$ and  define 
and $\S_1:=\S\cap \Omega_1\times \Omega_1$.  By the Lusin-Novikov Theorem \ref{thm:lusin-novikov} we can find Borel partitions $\S_1 = \bigcup_{n \in \NN} E_n$ and $\S_1 = \bigcup_{n \in \NN} F_n$ such that $\pi_l|_{E_n}$ and $\pi_r|_{F_n}$ are injective for all $n \in \NN$.  At least one of the intersections $E_n \cap F_m$ for $n,m \in \NN$ is non-negligible, so that we can pick a Borel subset $\S_2\subset \S_1$ with $\gamma_l(\S_2)>0$ and on which both $\pi_l$ and $\pi_r$ are injective.  We put $X_2':=\pi_l(\S_2)$ and $Y_2':=\pi_r(\S_2)$ and observe that $\pi_l|_{\S_2}: \S_2 \to X_2'$ and $\pi_r|_{\S_2}: \S_2 \to Y_2'$ are Borel isomorphisms by Lemma \ref{lem:injective-borel}.  Note also that
\[
0<\gamma_l(\S_2)=\int_{X_0'} |\pi_l^{-1}(x_0)\cap \S_2|d\mu_0'(x_0)=\mu_0'(X_2').
\]
 We may then define a Borel isomorphism $\psi:=\pi_r\circ (\pi_l\restriction_{\S_2})^{-1}\colon X_2'\to Y_2'$ and the formula \eqref{left-right-formula} now implies that
\[
\psi_*\left(\mu_0'\restriction_{X_2'}\right)=\frac{\covol(X_0)}{\covol(Y_0)}\nu_0'\restriction_{Y_2'},
\]
and, by definition, $\psi$ preserves the restrictions of the orbit equivalence relation of $K\curvearrowright \Omega$ to $X_2'$ and $Y_2'$, respectively.  Moreover, we have $K.X_2'=K.Y_2'$ 
and since $\mu_0'(X_2')>0$ it follows from Claim \ref{pf-formula-1} that this set is non-negligible in $\Omega$.  The proof is now concluded by applying a maximality argument:  Denote by $\mathbb{A}$ the set of countable families $(\S_i)$ of pairwise disjoint Borel subsets $\S_i \subset \S$ such that $\pi_l$ and $\pi_r$ are injective on $\bigcup_i \S_i$ and such that $K.\pi_l(\S_i)=K.\pi_r(\S_i)$ are pairwise disjoint and non-negligible subsets of $\Omega$.  Then $\mathbb{A}$ is inductively ordered by inclusion of families.  So we can pick a maximal element $(\S_{\max, i})$ and put $\S_\max = \bigcup_i \S_{\max, i}$.  Maximality implies that $X_\max':=\pi_l(\S_\max)$ and $Y_\max':=\pi_r(\S_\max)$ satisfy that $K.X_\max'=K.Y_\max'$ is conull in $\Omega$.
Moreover, by what was shown above the associated map $\psi_{\max}:=\pi_r \circ (\pi_l\restriction_{\S_\max})\colon X_\max' \to Y_{\max}'$ scales the restricted measures and maps orbits to orbits; thus $\psi_\max$ is an orbit equivalence if we can show that the saturations of $X_\max'$ and $Y_{\max}'$ are conull. 
To this end, note that the saturation of $X_\max'$ equals $K.X_\max'\cap X_0'$ and hence its complement in $X_0'$ is $X_0'\cap (\Omega\setminus K.X_\max)$. By construction $N:= \Omega\setminus K.X_\max$ is a $K$-invariant $\eta$-null set and applying Claim \ref{pf-formula-1} again we get
\begin{align*} 
\tfrac{1}{\covol(X_0)}(\lambda_K \times \mu_0')(K\times (X_0'\cap N)) &= \eta_{\Z_{X_0'}}\left( \{ (k.x_0,x_0)\mid k\in K, x_0\in X_0'\cap N     \} \right)\\
&=\int_{t\in \Omega} \left| \pi_l^{-1}(t)\cap  \{ (k.x_0,x_0)\mid k\in K, x_0\in X_0'\cap N     \}     \right| \mathrm{d} \eta(t)\\
&=\int_{t\in N} \left| \pi_l^{-1}(t)\cap  \{ (k.x_0,x_0)\mid k\in K, x_0\in X_0'\cap N     \}     \right| \mathrm{d}\eta(t)\\
&=0.
\end{align*}
Thus $\mu_0'(X_0'\cap N)=0$ and the proof is complete.
\end{proof}

The proof of the following lemma is inspired by the one of Theorem 2.3 in \cite{gaboriau}.

\begin{lem}
  \label{oe-implies-me}  
  If $G$ and $H$ are unimodular lcsc groups admitting orbit equivalent, essentially free, ergodic, pmp actions on standard Borel probability spaces then $G$ and $H$ are measure equivalent.  
\end{lem}

\begin{proof}
 Let $G\curvearrowright (X,\mu)$ and $H\curvearrowright (Y,\nu)$ be orbit equivalent, essentially free, ergodic, pmp actions. By \cite[Lemma 10]{stefaan-sven-niels} we may assume that both actions are genuinely free.    Choose an orbit equivalence $\Delta\colon X\to Y$ and denote by $X_0\subset X$ and $Y_0\subset Y$ the conull Borel subsets between which the restriction $\Delta_0:=\Delta\restriction_{X_0}\colon X_0 \to Y_0$ is a Borel isomorphism mapping orbits to orbits.  Denote by $\mathcal{R}_G = \{(gx, x) \in X \times X \mid g \in G , x \in X\}$ the orbit equivalence relation associated with the action $G \curvearrowright X$.  The Borel map $\alpha\colon G \times X \to X \times X$, $\alpha(g,x):=(gx,x)$, is injective, and hence a Borel isomorphism onto its image $\mathcal{R}_G$, by Lemma \ref{lem:injective-borel}.  Similarly, we denote the orbit equivalence relation of $H \curvearrowright Y$ by $\mathcal{R}_H$ and obtain a Borel isomorphism $\beta\colon H \times Y \to \mathcal{R}_H$.  Yet another application of Lemma \ref{lem:injective-borel} shows that the map $\Delta_0 \times \Delta_0$ restricts to a Borel isomorphism $\varphi\colon \mathcal{R}_G|_{X_0} \to \mathcal{R}_H|_{Y_0}$.  \\

 Consider the measure $\lambda_G \times \mu$ on $G \times X$ and denote its push-forward to $H \times Y$ by $\sigma = (\beta^{-1} \circ \varphi \circ \alpha)_*(\lambda_G \times \mu)$.  We will show that there is a Haar measure $\lambda_H$ of $H$ such that $\sigma = \lambda_H \times \nu$.  First observe that the equality $\pi_{Y_0} \circ \beta^{-1} \circ \varphi \circ \alpha = \Delta_0 \circ \pi_{X_0}$ on $\alpha^{-1}(\mathcal{R}_G|_{X_0})$ implies that $(\pi_Y)_*(\sigma) \sim \nu$.  Let $d\colon H \times Y \to G$ be the cocycle associated with $\Delta$ by Lemma~\ref{lem:cocycles} and denote by $\tilde d\colon H \times X \to G$ a strict cocycle that agrees with $d \circ (\id \times \Delta)$ almost everywhere, which exists by \cite[Theorem B.9]{zimmer-book}.  Define an $H$-action on $G \times X$ by $h * (g, x) = (\tilde d(h, gx) g, x)$.  For almost all $h \in H$ and almost all $(g,x) \in G \times X$ the equality
 \begin{gather*}
   (\beta^{-1} \circ \varphi \circ \alpha) (h * (g,x)) = h \bigl ( (\beta^{-1} \circ \varphi \circ \alpha)(g,x) \bigr )
 \end{gather*}
 is well-defined and holds.  For a positive measurable function $f: G \times X \to \mathbb{R}_{\geq 0}$, we have
 \begin{align*}
   \int_{G \times X} f(h*(g,x)) \mathrm{d}(\lambda_G \times \mu)(g,x)
   & =
     \int_{G \times X} f(\tilde d(h, gx)g,x) \mathrm{d}(\lambda_G \times \mu)(g,x) \\
   & =
     \int_{G \times X} f(\tilde d(h, x)g,g^{-1}x) \mathrm{d}(\lambda_G \times \mu)(g,x) \\
   & =
     \int_{G \times X} f(g,g^{-1}x) \mathrm{d}(\lambda_G \times \mu)(g,x) \\
   & =
     \int_{G \times X} f(g,x) \mathrm{d}(\lambda_G \times \mu)(g,x)
     \text{.}
 \end{align*}
Denoting by $\sigma = \int_Y \sigma^y \mathrm{d}\nu(y)$ the decomposition of $\sigma$ with respect to $\pi_Y$, the Fubini theorem combined with the previous calculation implies that for almost all $y \in Y$ the subgroup ${\{h \in H \mid h\sigma^y = \sigma^y\}}$ is conull in $H$.  It is hence equal to $H$ by \cite[Proposition B.1]{zimmer-book}.  So for almost all $y \in Y$ the measure $\sigma^y$ is $H$-invariant.  By uniqueness of the Haar measure, this implies that there is a Haar measure $\lambda_H$ on $H$ and a measurable function $b\colon Y \to \mathbb{R}_{> 0}$ such that $\sigma = \lambda_H \times b \nu$.  Making use of the action $(g,x) \mapsto (g, d(h, \Delta(x))x)$, a short calculation similar to the above-one shows that $b$ must be $H$-invariant.  By ergodicity it is hence almost surely constant, so that after rescaling the Haar measure on $H$, we obtain $\sigma = \lambda_H \times \nu$ as desired. \\

Let us now consider $\mathcal{R}  = (\id_{X_0} \times \Delta_0)(\mathcal{R}_G|_{X_0}) = (\Delta_0^{-1} \times \id_{Y_0})(\mathcal{R}_H|_{Y_0}) \subset X_0 \times Y_0$.  On $\mathcal{R}$ we define a near action of $G \times H$ by the formulas
 \begin{gather*}
     g (x,y) =
   \begin{cases}
     (gx,y) & \text{if } gx \in X_0,  \\
     (x,y) & \text{otherwise}
   \end{cases}
 \end{gather*}
 and 
 \begin{gather*}
     h (x,y) =
   \begin{cases}
     (x,hy) & \text{if } hy \in Y_0,  \\
     (x,y) & \text{otherwise.}
   \end{cases}
 \end{gather*}
 Note that both formulas define Borel maps since $\{(g,x) \in G \times X_0 \mid g x \in X_0\}$ is a Borel subset of $G \times X_0$ as it is the image of $\mathcal{R}_G|_{X_0}$ and similarly $\{(h,y ) \in H \times Y_0 \mid hy \in Y_0\}$ is a Borel subset of $H \times Y_0$.  By Lemma \ref{lem:near-action-to-Borel-action}, there is a standard Borel $G \times H$-space $\Omega$ and a measure space isomorphism $\psi\colon \Omega \to \mathcal{R}$ such that for all $(g,h) \in G \times H$ the equality $(g,h)\psi(\omega) = \psi((g,h)\omega)$ holds for almost all $\omega \in \Omega$.  We equip $\Omega$ with the measure 
 \[
 \eta := (\psi^{-1} \circ  (\id_{X_0}\times \Delta_0)\circ \alpha)_*(\lambda_G \times \mu) = (\psi^{-1} \circ (\Delta_0^{-1}\times \id_{Y_0})\circ\beta)(\lambda_H \times \nu).
\]
 The composition of the measure space isomorphisms $i = \alpha^{-1} \circ (\id_{X_0}\times \Delta_0^{-1})\circ \psi \colon \Omega \to G \times X$ induces a $G$-equivariant isomorphism on measure algebras $B(\Omega, \eta) \to B(G \times X, \lambda_G \times \mu)$, which by \cite[Theorem 2]{mackey-point-realization} implies that it is an isomorphism of measure $G$-spaces $\Omega \to G \times X$.  Similarly, the composition $j = \beta^{-1} \circ \psi \colon \Omega \to H \times Y_0$ is an isomorphism of measure $H$-spaces.  This shows that $(\Omega, \eta, X, \mu, Y, \nu, i, j)$ is a measure equivalence coupling between $(G, \lambda_G)$ and $(H, \lambda_H)$, which finishes the proof. 
\end{proof}

Finally, we gather the results above into a proof of Theorem \ref{thm:me-and-soe-equivalence}.
\begin{proof}[Proof of Theorem \ref{thm:me-and-soe-equivalence}]
Assuming $G$ and $H$ measure equivalent, we may, by Proposition \ref{prop:free-ergodic-coupling},  choose a strict, free, ergodic measure coupling $(\Omega, \eta, X,\mu, Y,\nu, i,j)$. The induced actions $G\curvearrowright (X,\mu)$ and $H\curvearrowright (Y,\nu)$ are then also  free and ergodic  and by Lemma \ref{lem:ME-implies-SOE-cross} it follows that the cross section equivalence relations associated with any choice of cross sections $X_0\subset X$ and $Y_0\subset Y$ are stably orbit equivalent; thus (i) implies (iii).
 Assuming (iii), we have essentially free, ergodic, pmp actions $G\curvearrowright (X,\mu)$ and $H\curvearrowright (Y,\nu)$ with cross sections $X_0$ and $Y_0$ for which the associated cross section equivalence relations, $\R_{X_0}$ and $\R_{Y_0}$, are stably orbit equivalent. 
 Since $\mathrm S^1\curvearrowright (\mathrm{S^1},\lambda_{\mathrm{S^1}})$ is free and ergodic so is the diagonal action $G\times \mathrm{S^1}\curvearrowright  X\times \mathrm{S^1}$ (freeness is clear and any invariant subset is of the form $A\times \mathrm{S^1}$ for a $G$-invariant subset $A$ in  $X$) and $X_0\times \{1\}$ is a cross section for this action and the associated cross section equivalence relation is orbit equivalent with $\R_{X_0}$. So, we obtain essentially free, ergodic, pmp actions of $G\times \mathrm{S^1}$ and $H\times \mathrm{S^1}$ with stably orbit equivalent cross section equivalence relation and by Lemma \ref{lem:SOE-cross-eq-implies-OE-actions} we conclude that the original actions are orbit equivalent; this shows that (iii) implies (ii). Lastly, assuming (ii) we get  that $G \times \mathrm{S^1}$ is measure equivalent with $H\times \mathrm{S^1}$ from Lemma  \ref{oe-implies-me}, and hence (i) holds because any unimodular lcsc group $G$ is a cocompact subgroup of $G\times \mathrm S^1$ and in particular measure equivalent to it.
 
Finally we need to address the two special cases: If $G$ and $H$ are discrete, the equivalence of (i) and (ii)' is well known \cite{Furman-OE-rigidity} and the equivalence of (ii)' and (iii) is obvious as the whole space is a cross section for an essentially free action of a countable, discrete group. If both $G$ and $H$ are non-discrete, the proof given above goes through without passing to an amplification with $\mathrm{S^1}$ before applying Lemma \ref{lem:SOE-cross-eq-implies-OE-actions}, thus showing the equivalence between  (i), (ii)''  and (iii).
\end{proof}

\section{Measure equivalence of amenable lcsc groups}
\label{sec:me-amenable-groups}

In this section we prove Theorem \ref{thm:amenable-grps-are-me-intro} and show,  conversely,  that if a unimodular lcsc  group is measure equivalent to a unimodular amenable group, then it is itself amenable. 

\begin{thm}
  \label{thm:amenable-groups-are-ME}
  All non-compact, amenable, unimodular, lcsc groups are pairwise measure equivalent.
\end{thm}
\begin{proof}
By the Ornstein-Weiss theorem \cite[Theorem 6]{ornstein-weiss} (see also \cite{Furman-OE-rigidity, Dye-mp-trans-I, Dye-mp-trans-II}) all infinite, discrete, countable, amenable groups are measure equivalent, and since $\ZZ$ is measure equivalent to $\RR$ we only need to prove that any pair of non-discrete, non-compact amenable, lcsc unimodular groups $G$ and $H$ are measure equivalent.  Let therefore $G$ and $H$ be two such groups  and pick essentially free, ergodic pmp actions $G\curvearrowright (X,\mu)$, $H\curvearrowright (Y,\nu)$; cf.~\cite[Proposition 1.2]{adams-elliot-giordano} and \cite[Remark 1.1]{KPV} for the existence of such actions.  Choose cross sections $X_0\subset X$ and  $Y_0\subset Y$ for the two actions and recall that the induced cross section equivalence relations are ergodic and amenable \cite[Proposition 4.3]{KPV}, and hence  orbit equivalent by \cite{Connes-Feldman-Weiss} and \cite{Dye-mp-trans-I}.
Hence $G$ and $H$ are measure equivalent by Theorem \ref{thm:me-and-soe-equivalence}.
\end{proof}

The techniques used in the proof above also provides an explicit proof of the following well-known consequence of \cite{Connes-Feldman-Weiss}  (cf.~\cite[Theorem B]{bowen-hoff-ioana} for a converse statement): all essentially free, ergodic, probability measure preserving  actions of non-compact, non-discrete, amenable, unimodular, locally compact, second countable groups on standard Borel probability spaces are pairwise orbit equivalent. Namely, given two such groups $G$ and $H$ with essentially free, ergodic, pmp action on standard Borel probability spaces $(X,\mu)$ and $(Y,\nu)$, respectively, then, as in the proof just given we conclude that the cross section equivalence relations associated with any choice of cross sections for the two actions are orbit equivalent, and hence the original actions are orbit equivalent by Lemma \ref{lem:SOE-cross-eq-implies-OE-actions}.


\begin{thm}  \label{thm:amenability-ME-invariant}
If $G$ and $H$ are measure equivalent,  unimodular,  lcsc groups and one of them is amenable then so is the other.
\end{thm}
\begin{proof}
Invoking Proposition \ref{prop:free-ergodic-coupling},  we may choose a strict, free and ergodic measure coupling $(\Omega, \eta, X, \mu, Y,\nu, i,j)$ for which  the induced actions $G\curvearrowright (X,\mu)$ and $H\curvearrowright (Y,\nu)$ are therefore free and ergodic as well. By Lemma \ref{lem:ME-implies-SOE-cross}, we moreover have that for any choice of cross sections $X_0\subset X$ and $Y_0\subset Y$ the associated cross section equivalence relations are stably orbit equivalent. Denoting by $\mu_0$ and $\nu_0$ the measures on $X_0$ and $Y_0$ given by Proposition \ref{prop:cross-section}, we may therefore find  Borel subsets $A\subset X_0$ and $B\subset Y_0$ that are non-negligible with respect to $\mu_0$ and $\nu_0$, respectively, and such that the restricted orbit equivalence relation $\R_{X_0}\restriction_A$ is orbit equivalent with $\R_{Y_0}\restriction_B$. However, since the actions $G\curvearrowright (X,\mu)$ and $H\curvearrowright (Y,\nu)$ are ergodic, the sets $A$ and $B$ are also cross sections  and as the following computation shows,  the probability measure $\mu_A$ associated with $A$ by proposition Proposition \ref{prop:cross-section} is just a re-scaling of $\mu_0\restriction_A$: choosing an identity neighbourhood $U$ for $X_0$ as in Proposition \ref{prop:cross-section} and a Borel set $E\subset A$ we get
\[
\lambda_G(U)\mu_A(E)= \covol(A)\mu(U.E)=\frac{\covol(A)}{\covol(X_0)}\lambda_G(U)\mu_0(E).
\]
Since $\mu_A$ is a probability measure, it must be equal to $\mu(A)^{-1}\mu_0\restriction_A$ (and similarly for $B$). Now, by Proposition \ref{prop:cross-section} we have that $G$ is amenable if and only if the cross section equivalence relation $\R_{A}$ associated with $A$ is amenable and $H$ is amenable if and only if $\R_B$ is amenable and since $\R_A$ and $\R_B$ are orbit equivalent they are amenable (cf.~\cite[Definition 6]{Connes-Feldman-Weiss}) simultaneously.
\end{proof}

\section{Measure equivalence and property (T)}
\label{sec:T}
The aim of this section is to provide a proof of Theorem \ref{thm:me-and-property-T}, stating that property (T) is preserved under measure equivalence. This result may be known to experts in the field, but to the best of our knowledge has not been stated or proven explicitly anywhere. The proof follows the rough outline of the corresponding proof for discrete groups presented in \cite[Theorem  6.3.13]{BHV}, with a few additional measure theoretical wrinkles. Note  that groups with property (T) are automatically unimodular \cite[Corollary 1.3.6]{BHV}.

\begin{proof}[Proof of Theorem  \ref{thm:me-and-property-T}]
Assume towards a contradiction that $G$ has property (T) and $H$ does not. Then $H$ admits a strongly continuous  unitary representation $\pi\colon H\to \mathrm{U}(\H)$ with almost invariant vectors which does not contain any finite-dimensional subrepresentations; see \cite[Remark 2.12.11]{BHV} for this. 
Before embarking on the actual proof, we first show that the space $\H$ can be chosen separable.
  Since $H$ is lcsc, thus in particular hemicompact, we may choose an increasing sequence $K_n\subset H$ of compact subsets that is cofinal in the family of all compact subsets and has union $H$, and since $\pi$ has almost invariant vectors, for each $K_n$ there exists a unit vector $\xi_n\in \H$ such that $\|\pi(h).\xi_n-\xi_n\|<\tfrac{1}{n}$ for all $h\in K_n$. As $H$ is separable, the subspace
\[
\H_0:=\overline{\text{span}}\left\{ \pi(h).\xi_n \mid h\in H, n\in \NN \right\} \subset \H
\]
is separable and $H$-invariant and the restriction of $\pi$ to $\H_0$ still has almost invariant vectors and of course no finite dimensional subrepresentations since this was already the case for $\pi$. Thus, by replacing $\pi$ with its subrepresentation on $\H_0$ we may as well assume that $\H$ is separable.
Let $(\Omega,\eta, X,\mu, Y, \nu, i,j)$ be measure equivalence coupling between $G$ and $H$ and assume, as we may by Theorem \ref{thm:strict-coupling}, that $\Omega$ is strict, and denote by $\omega_G\colon H\times Y\to G $ and $\omega_H\colon G\times X \to H$ the associated (strict) measurable cocycles.  The representation $\pi$ can be induced up to a unitary representation $\tilde{\pi}$ of $G$ on $L^2(X,\H)$  given by
\[
\tilde{\pi}(g)(\xi)(x):=\pi\left(\omega_H(g^{-1},x)^{-1}\right)\xi(g^{-1}.x)= \pi\left(\omega_H(g,g^{-1}.x)\right)\xi(g^{-1}.x).
\]
Since $\omega_H$ is a measurable map, it follows that $g\mapsto \ip{\tilde{\pi}(g)\xi}{\xi}$ is measurable for all $\xi\in L^2(X,\H)$ and since  $(X,\mu)$ is standard Borel probability space and $\H$ is separable, $L^2(X, \H)$ is also separable. Lemma A.6.2 in \cite{BHV} therefore applies to show that
 $\tilde{\pi}$ is  a strongly continuous unitary representation. We now prove that the induced representation $\tilde{\pi}$ also has almost invariant vectors. To this end, note first that by \cite[Theorem 4.1]{deprez-li-ME} the cocycle $\omega_H\colon G\times X\to H$ is \emph{proper}  in the sense of \cite[Definition 2.2]{deprez-li-ME}, meaning  that there exists a family $\mathscr{A}$ of Borel sets in $X$ with, among others, the following two properties:
\begin{enumerate}
\item For every compact set $K\subset G$ and every $A,B\in \mathscr{A} $ there exists a precompact set $L\subset H$ such that for all $g\in K$
$
\omega_H(g,x)\in L \ \text{  and almost all $x\in A\cap g^{-1}.B$}.
$
\item For every $\varps>0$ there exists $A\in \mathscr{A}$ such that $\mu(X\setminus A)<\varps$.
\end{enumerate}
To prove that $\tilde{\pi}$ has invariant vectors,  let a compact subset $K\subset G$ and $\varps>0$ be given and choose a set $A\in \mathscr {A} $ with $\mu(X\setminus A)<\varps $ and a precompact set $L\subset H$ satisfying (1) with respect to the given compact set $K$ and $B=A$. Since $\pi$ is assumed to have almost invariant vectors, there exists a unit vector $\eta\in \H$ such that $\|\pi(h)\eta-\eta\|<\varps$ for all $h\in L$, and since $\mu$ is finite, the function $\tilde{\eta}(y):=\tfrac{1}{\sqrt{\mu(X)}}\eta$ is a unit vector in in $L^2(X,\H)$.  For $g\in K$ we therefore have
\begin{align*}
\|\tilde{\pi}(g) \tilde{\eta}- \tilde{\eta}  \|^2 &= \int_{X} \| \tilde{\pi}(g)\tilde{\eta}(x)-\tilde{\eta}(x)  \|^2 \mathrm{d}\mu(x)\\
&= \tfrac{1}{\mu(X)}\int_{X} \| \pi(\omega_H(g,g^{-1}.x))\eta- \eta  \|^2 \mathrm{d}\mu(x) \\
&=  \tfrac{1}{\mu(X)}\int_{A\cap g.A} \| \pi(\omega_H(g,g^{-1}.x))\eta- \eta  \|^2 \mathrm{d}\mu(x)  \\
&+ \tfrac{1}{\mu(X)} \int_{X\setminus (A\cap g.A)} \| \pi(\omega_H(g,g^{-1}.x))\eta- \eta  \|^2 \mathrm{d}\mu(x)\\
&\leq \varps^2 + \tfrac{1}{\mu(X)}4\mu(X\setminus A) +\tfrac{1}{\mu(X)}4\mu(X\setminus g.A)\\
&< \varps^2 + \tfrac{8\varps}{\mu(X)}.
\end{align*} 
As $\varps>0$ was arbitrary this shows that $\tilde{\pi}$ has almost invariant vectors and since $G$ is assumed to have property (T), this means that  $\tilde{\pi}$ must have a non-trivial invariant vector $ L^2(X,\H)$. Thus, any Borel representative $\xi_0\colon X \to \H$ for such an invariant vector satisfies that for all $g\in G$
\begin{align}\label{eq:tilde-xi-inv}
\pi\left(\omega_H\left(g^{-1}, x\right)^{-1}\right)\xi_0(g^{-1}.x)=\xi_0(x) \ \text{ for almost all } x\in X.
\end{align}
Consider now the unitary representation $\rho\colon H\to \mathrm{U}(L^2(Y,\H))$ given by
\[
\rho(h)(\zeta)(y):=\pi(h)\zeta(h^{-1}.y).
\]
The representation $\rho$ is unitarily equivalent with $\lambda_Y\otimes \pi$ on $L^2(Y)\bar{\otimes} \H$, where $\lambda_Y$ is the unitary representation induced by the measure preserving action $H\curvearrowright Y$. To reach the desired contradiction, we now prove that $\rho$ has a non-trivial invariant vector.  Indeed by \cite[Proposition A.1.12]{BHV}, this implies that $\pi$ contains a finite dimensional subrepresentation contradicting its defining properties.
Consider again the representative $\xi_0$ for the vector $L^2(X,\H)$ fixed by $\tilde{\pi}$ and extend $\xi_0$ to a Borel map $\tilde{\xi}_0\colon H\times X\to \H$ by setting
$
\tilde{\xi}_0(h,x):=\pi(h)\xi_0(x).
$
Note that for any $h_0\in H$ we have
\begin{align*}\label{eq:H-equiv-of-xi-zero}
\tilde{\xi}_0(h_0h,x)=\pi(h_0)\tilde{\xi}_0(h,x).
\end{align*}
Considering the $G$-action on $H\times X$ induced by $j^{-1}$ we therefore have that for all $g\in G$, all $h\in H$ and almost all $x\in X$
\[
\tilde{\xi}_0\left(g.(h,x)\right)=\tilde{\xi}_0\left(h\omega_H(g,x)^{-1}, g.x\right)=\pi(h)\tilde{\xi}_0\left(\omega_H(g,x)^{-1},g.x\right) 
\xlongequal{\eqref{eq:tilde-xi-inv}}
\pi(h)\xi_0(x)=\tilde{\xi}_0(h,x),
\]
Since $\H$ is separable it is, in particular, a Polish space and hence a standard Borel space with respect to its Borel $\sigma$-algebra. Moreover, for any $\eta\in\H $ the map $H\times \H\to \CC$ given by $(h,\xi)\mapsto \ip{\pi(h)\xi}{\eta}=\ip{\xi}{\pi(h^{-1})\eta}$ is jointly continuous, showing that the action map $H\times \H\to \H$ is weakly measurable, and thus measurable by Pettis' measurability theorem \cite[Theorem 1.1]{pettis}. 
Hence,  $\H$ is a standard Borel $G\times H$-space for the action $(g_0,h_0).\xi:=\pi(h_0)\xi$ and considering $H\times X$ as a $G\times H$-space for the action
\[
(g_0,h_0).(h,x):=(h_0h\omega_H(g_0,x)^{-1}, g_0.x),
\]
we obtain that the function $\tilde{\xi}_0\colon H\times X \to \H$ is almost everywhere $G\times H$-equivariant.
An application of \cite[Proposition B.5]{zimmer-book} provides us with a $G\times H$-invariant, conull, Borel subset $S \subset H\times X$ and a measurable function $\tilde{\xi}_0' \colon S \to \H $ which is genuinely $G\times H$-equivariant and agrees with $\tilde{\xi}_0$ almost everywhere. From $G\times H$-invariance of $S$ it follows that it has the form $S=H\times X_0$ for a $G$-invariant, conull, Borel subset $X_0\subset X$. Thus, by extending $\tilde{\xi}_0'$ by zero on $H\times (X\setminus X_0)$, we obtain  a measurable  map  $\tilde{\xi}_0'\colon H\times X\to \H $ such that
\begin{align*}
\tilde{\xi}_0' (h_0h,x)&=\pi(h_0)\tilde{\xi}_0'(h,x) \ \text{ for all }h_0\in H, (h, x)\in H\times  X,\\
\tilde{\xi}_0'\left(h\omega_H(g_0,x)^{-1}, g_0.x\right) &=\tilde{\xi}_0'(h,x) \ \text{ for all } g_0\in G, (h, x)\in H\times  X.
\end{align*}
The function $\tilde{\xi}_0'':=\tilde{\xi}_0'\circ j^{-1}\circ i \colon G\times Y \to \H$ therefore satisfies
\begin{align}
\tilde{\xi}_0''(g_0g,y)&=\tilde{\xi}_0''(g,y) \ \text{ for all } g_0\in G,(g,y)\in G\times Y, \label{eq:G-inv}\\
\tilde{\xi}_0''(g\omega_G(h_0,y)^{-1}, h_0.y) &=\pi(h_0)\tilde{\xi}_0''(g,y) \ \text{ for all } h_0\in H, (g,y)\in G\times Y \label{eq:H-inv}.  
\end{align}
Since $\xi_0\neq 0$, the function $\tilde{\xi}_0''$ cannot be zero almost everywhere and combining this with the $G$-equivariance \eqref{eq:G-inv} we infer that the Borel map $\zeta\colon Y\to \H $ given by $\zeta(y):= \tilde{\xi}_0''(e_G,y)$ is not $\nu$-almost everywhere zero. Hence, for a suitable choice of $M,\delta>0$, the set 
\[
Y_0:=\{y\in Y \mid \delta<\|\zeta(y)\|< M\}
\]
is $\nu$-non-negligible.
Moreover, for $h\in H$ and $y\in Y$ we have

\begin{align*}
\zeta(h.y)=\tilde{\xi}_0''(e_G,h.y) \xlongequal{\eqref{eq:G-inv}} \tilde{\xi}_0''\left(\omega_G(h,y)^{-1},h.y\right)  \xlongequal{\eqref{eq:H-inv}}\pi(h)\tilde{\xi}_0''(e_G,y) 
=\pi(h)\zeta(y),
\end{align*}
and since $\pi$ is a unitary representation it follows that $Y_0$ is $H$-invariant. Thus, setting $\zeta_0:=\bbb_{Y_0}\zeta$ we have that $\zeta_0\in L^2(Y,\H)\setminus \{0\} $.  By the previous computation it follows that
\[
\zeta_0(h.y)=\bbb_{Y_0}(h.y)\zeta(h.y)=\bbb_{h^{-1}.Y_0}(y)\pi(h)\zeta(y)= \bbb_{Y_0}(y) \pi(h)\zeta(y)= \pi(h)\zeta_0(y),
\]
showing that $\zeta_0$ is invariant for the representation $\rho$.  By the remarks above, this finishes the proof.
\end{proof}

\section{Uniform measure equivalence}\label{sec:uniform-me}

Measure equivalence can, via Gromov's dynamic criterion discussed below,  be seen as  a measure theoretic analogue of coarse equivalence, although, in general, neither of these notions implies the other (cf.~\cite{carette-haagerup}).
 Heuristically,  uniform measure equivalence should provide a notion of simultaneous measure equivalence and coarse equivalence, but one should note that it is not true that two (even discrete) groups that are measure equivalent and coarsely equivalent  are automatically uniformly measure equivalent, as \cite{das-tessera-integrable-me} shows. In this {section, we} provide a definition of uniform measure equivalence for unimodular lcsc groups extending the existing definition for finitely generated discrete groups. We  begin by recalling the notion of coarse equivalence, following the presentation in \cite{roe-lecture-notes}.

\begin{definition}
  \label{def:coarse-space}
  Let $X$ be a set.  A \emph{coarse structure} on $X$ is a collection of subsets $\mathcal E \subset \mathcal P(X \times X)$ called \emph{controlled sets} such that the following requirements are satisfied{:}
  \begin{itemize}
  \item[(i)] {t}he diagonal is controlled{;}
  \item[(ii)] {a} subset of a controlled set is controlled{;}  
  \item[(iii)] {a} finite union of controlled sets is controlled{;}
  \item[(iv)] {i}f $E \in \mathcal E$ is controlled, then so is $\{(x,y) \in X \times X \mid (y,x) \in E\}${;}
  \item[(v)] {i}f $E_1, E_2 \in \mathcal E$, then so is $\{(x, z) \in X \times X \mid \exists y \in X: (x,y) \in E_1, (y, z) \in E_2\} ${.}
  \end{itemize}
  A set equipped with a coarse structure is called a \emph{coarse space}.
\end{definition}

\begin{definition}
  \label{def:bounded-cobounded}
  Let $X$ be a coarse space.
  \begin{itemize}
  \item[(i)] A subset $B \subset X$ is called \emph{bounded} if $B \times B$ is controlled.
  \item[(ii)] A subset $C \subset X$ is \emph{cobounded} if there is a controlled set $E$ such that for all $x\in X$ there exists $c\in C$ such that $(c,x)\in E$.
  \end{itemize}
\end{definition}

\begin{definition}
  \label{def:coarse-maps}
  Let $f\colon X \rightarrow Y$ be a map between coarse spaces.  Then $f$ is \emph{bornologous} if $f \times f$ maps controlled sets to controlled sets and $f$ is \emph{proper} if preimages of bounded sets  are bounded.  If $f$ is bornologous and proper, then it is said to be a \emph{coarse map}.  Further, a \emph{coarse embedding} is a map $f\colon X \rightarrow Y$ such that $E \subset X \times X$ is controlled if and only if $(f \times f)({E}) \subset Y \times Y$ is controlled.  Finally, $f$ is a \emph{coarse equivalence} if it is a coarse map and there is a coarse map $g\colon Y \rightarrow X$ such that $\{(x, g \circ f(x)) \in X \times X \mid x \in X\}$ and $\{(y, f \circ g(y) \in Y \times Y \mid y \in Y\}$ are controlled.
\end{definition}

\begin{remark}
  \label{rem:characterisation-coarse-equivalence}
    Coarse equivalences between two coarse spaces $X$ and $Y$ can be characterized as those coarse embeddings $f: X \rightarrow Y$ whose image is cobounded.
\end{remark}

Let us describe the main examples of coarse spaces relevant in the context of uniform measure equivalence.
  \begin{enumerate}
  \item If $G$ is a locally compact group it admits two natural coarse structures: 
   The controlled sets of the {\emph{left coarse structure}} on $G$ are all subsets of sets of the form
    \begin{equation*}
       \{(g,h) \in G \times G \mid g^{-1}h \in K\}
       =
       \{(g,h) \in G \times G \mid h \in g K\}
       \text{,}
    \end{equation*}
    where $K$ runs through all compact subsets of $G$. The controlled sets of the {\emph{right coarse structure}} on $G$ are all subsets of sets of the form
    \begin{equation*}
       \{(g,h) \in G \times G \mid hg^{-1} \in K\}
       =
       \{(g,h) \in G \times G \mid h \in Kg \}
       \text{,}
    \end{equation*}
    where again $K$ runs through all compact subsets of $G$.   The inversion of $G$ is a bijective coarse equivalence between these two coarse structures.  Note that the left coarse structure of $G$ is left-invariant, while the right coarse structure of $G$ is right-invariant.
    
  \item Any set $X$ can be turned into a bounded coarse space, by declaring all subsets of $X \times X$ controlled.
  \item If $X$, $Y$ are coarse spaces, then $X \times Y$ is a coarse space whose controlled sets are those subsets of $X \times Y \times X \times Y$ whose projection to $X \times X$ and $Y \times Y$ are controlled.  
 \item If $(X,d)$ is a metric space then the \emph{bounded coarse structure} on $X$ is defined by declaring a set $E\subset X\times X$ controlled if
 \[
 \sup\{d(x,y)\mid x,y\in E\}<\infty.
 \]
 In this situation, a coarse equivalence between metric spaces $(X,d_X)$ and $(Y,d_Y)$ can be described in terms of the metrics as a map $f\colon X\to Y$ with cobounded image ({i.e. $\sup_{y\in Y}d(f(X),y)<\infty$}) which satisfies that
 \[
 \lim_{n\to \infty} d_X(x_n,x_n')=\infty \quad \textrm{{if and only if}} \quad \lim_{n\to \infty }d_Y(f(x_n), f(x_n'))=\infty,
 \]
 for all sequences $(x_n), (x_n')$ in $X$.
 
 \end{enumerate}
If $G$ and $H$ are lcsc groups we say that they are coarsely equivalent  if if they are so when endowed with their left (equivalently right) coarse structures.
Equivalently, $G$ and $H$ are coarsely equivalent if they are so as metric spaces when both are endowed with any proper, compatible, left (equivalently right) invariant metric. Recall that a metric on $G$ is \emph{compatible} if it induces the original topology on $G$ and \emph{proper} if closed balls are compact, and that {any lcsc group admits a compatible, proper and left invariant  metric \cite{Struble}}. 
Note also that for compactly generated  lcsc groups, coarse equivalence coincides with the notion of {quasi-isometry with} respect to the {word metrics} arising from some/{any} compact generating sets \cite[Proposition 4.B.10]{yves-book}.

\begin{defi}\label{def:UME}
Let $G$ and $H$ be unimodular lcsc groups. A strict {measure equivalence} $G$-$H${-}coupling $(\Omega, \eta, X, \mu, Y, \nu, i,j) $ is said to be \emph{uniform} if the compositions $i^{-1}\circ j$ and $j^{-1}\circ i$ are proper maps {with respect to the product coarse structure} when $G$ and $H$ are endowed with the left coarse structure and $X$ and $Y$ are declared bounded. A {measure equivalence} $G$-$H${-}coupling $(\Omega, \eta, X, \mu, Y, \nu, i, j) $ is said to be uniform if there exist conull Borel subsets $\Omega_0\subset \Omega$, $X_0\subset X$ and $Y_0\subset Y$ such that $(\Omega_0, \eta, X_0, \mu, Y_0, \mu, \nu, i\restriction_{G\times Y_0}, j\restriction_{H\times X_0}) $ is a strict uniform measure equivalence $G$-$H$-coupling.  Lastly, $G$ and $H$ are said to be {uniformly measure equivalent if} they admit a uniform measure equivalence coupling.
\end{defi}

Since bounded sets in $G$ and $H$ are simply pre-compact sets, {cf. \cite[Example 2.24]{roe-lecture-notes}}, the condition that a strict {measure equivalence} $G$-$H$-coupling be uniform simply amounts to the following:
\begin{itemize}
\item {f}or every compact $C\subset G$ there exists a compact $D\subset H$ such that $j^{-1}\circ i(C\times Y)\subset D\times X${;} 
\item {f}or every compact $D\subset H$ there exists a compact $C\subset G$ such that $i^{-1}\circ j(D\times X)\subset C\times Y$.
\end{itemize}

\begin{lem}
Uniform measure equivalence is an equivalence relation
\end{lem}
\begin{proof}
Fixing a Haar measure $\lambda_G$ on $G$, it is clear that $(G,\lambda_G)$ is a strict uniform $G$-$G$ measure equivalence coupling when endowed with the action $(g_1,g_2).g:=g_1gg_2^{-1}$, thus showing reflexivity of the relation. If $(\Omega,\eta)$ is a uniform $G$-$H$ measure equivalence coupling then the dual coupling $(\overline{\Omega},\bar{\eta})$, which as a measure space is identical to $(\Omega,\eta)$ {but with the} $H\times G$-action $(h,g)\triangleright t:=(g,h).t$, is uniform as well. Hence, {the relation is symmetric}.
To show transitivity, we recall the composition of measure equivalence couplings as described in Section A.1.3 of \cite{BFS-integrable}{:  if} $G_1, G_2, H$ are {unimodular lcsc groups}, and $(\Omega_1, \eta_1, X_1,\mu_1,Y_1,\nu_1,i_1,j_1)$ is a (strict uniform) measure equivalence $G_1$-$H$-coupling and $(\Omega_2,\eta_2, Y_2,\mu_2,X_2,\nu_2,i_2,j_2) $ is a (strict uniform) measure equivalence $H$-$G_2$-coupling, then their composition is defined as $(\Omega,\eta):=(H\times X_1\times X_2, \lambda_H\times \mu_1\times \mu_2)$ with $G_1\times G_2$-action given by
\[
(g_1,g_2). (h,x_1,x_2):=\left(\omega_H^{(2)}(g_2,x_2)h\omega_H^{(1)}(g_1,x_1)^{-1},g_1.x_1,g_2.x_2 \right),
\]
where $\omega_H^{(1)}$ and $\omega_H^{(2)}$ are the cocycles associated with $\Omega_1$ and $\Omega_2$, respectively.
By Section A.1.3 of \cite{BFS-integrable}, the following maps witness that $\Omega$ is indeed a $G_1$-$G_2$ measure equivalence coupling:
\begin{equation*}
  \xymatrix@C+6pc
  {
    G_1 \times Y_1 \times X_2 \ar[r]^{(j_1^{-1} \circ \, i_1) \times \id_{X_2}} &
    H \times X_1 \times X_2 &
    G_2 \times Y_2 \times X_1 \ar[l]_{  (\text{inv} \times \sigma) \circ (i_2^{-1} \circ \, j_2\times \id_{X_1})}
    \text{,}
  }
\end{equation*}
where $\text{inv}$ denotes inversion in $H$ and $\sigma$ denotes the coordinate flip on $X_1\times X_2$.
Assuming now that both $\Omega_1$ and $\Omega_2$ are {strict uniform} measure equivalence couplings then both $j_1^{-1} \circ \, i_1$ and $i_2^{-1} \circ \, j_2$ as well as their inverses are proper maps, and {{since $(\sigma \times \text{inv})$ is proper as well,} $\Omega$} is also a strict uniform measure equivalence coupling. 
 \end{proof}

\begin{lem}\label{lem:cocycles-loc-bd}
  A strict measure equivalence {$G$-$H$-}coupling $(\Omega, \eta, X, \mu, Y, \nu, i,j) $ is uniform  if and only if
  \begin{itemize}
  \item the sets $i^{-1} \circ j(\{e_H\} \times X) \subset G \times Y$ and $j^{-1} \circ i (\{e_G\} \times Y) \subset H \times X$ are bounded, and 
  \item the associated cocycles {$\omega_G$ and $\omega_H$} are locally bounded; i.e.~for every compact {subset} $C\subset G$ there exists a compact {subset} $D\subset H$ such that $\{\omega_H(g,x)\mid g\in C, x\in X\}\subset D$ and similarly for $\omega_G$.
  \end{itemize}
\end{lem}

\begin{proof}
  First assume that $(\Omega, \eta, X, \mu, Y, \nu, i,j) $ is a strict uniform measure equivalence {$G$-$H$-}coupling.  Then the sets $i^{-1} \circ j(\{e_H\} \times X) \subset G \times Y$ and $j^{-1} \circ i (\{e_G\} \times Y) \subset H \times X$ are bounded by assumption.  Further, let a compact subset $C \subset G$ be given and fix $g \in C$ and $x \in X$. Then we have
\begin{align*}
\left(\omega_H(g,x)^{-1}, g.x\right) =j^{-1}\left(g.j(e_H,x)\right)
&= \left(j^{-1}\circ i\right)\left(g.i^{-1}\circ j\right)(e_H,x).
\end{align*}
Since $(i^{-1}\circ j)^{-1} = j^{-1} \circ i$ is proper there exists a compact subset  $C'\subset G$ such that ${i^{-1}\circ j}(\{e_H\}\times X)\subset C' \times Y$ and hence
$
g.(i^{-1}\circ j)(e_H,x) \in CC'\times Y.
$
Since $(j^{-1}\circ i)^{-1} = i^{-1} \circ j$ is also proper, there exists a compact subset $D\subset H$ such that ${j^{-1}\circ i}(CC'\times Y)\subset D\times X$ and we therefore obtain that $\{\omega_H(g,x)\mid g\in C, x\in X\}\subset D^{-1}$ as desired. The similar claim about $\omega_G$ follows by symmetry. So we proved that the cocycles $\omega_G$ and $\omega_H$ are locally bounded.

Vice versa, let us assume that the sets $i^{-1} \circ j(\{e_H\} \times X) \subset G \times Y$ and $j^{-1} \circ i (\{e_G\} \times Y) \subset H \times X$ are bounded and the associated cocycles {$\omega_G$ and $\omega_H$} are locally bounded.  Let $C \subset G$ be a compact set and find a compact subset $D \subset H$ satisfying $\omega_H(C, X) \subset D$.  Let $D' \subset H$ be a compact subset such that $(j^{-1} \circ i)(\{e_G\} \times Y) \subset D' \times X$.  Then
\begin{gather*}
  (j^{-1} \circ i)(C \times Y)
  \subset
  \omega_H(C, X) (j^{-1} \circ i)(\{e_G\} \times Y)
  \subset
  D D' \times X
\end{gather*}
shows that $i \circ j^{-1} = (j^{-1} \circ i)^{-1}$ is proper.  Properness of $j \circ i^{-1}$ follows by symmetry.
\end{proof}

\begin{remark}
{The above lemma shows that} the notion of {uniform measure equivalence} introduced in Definition \ref{def:UME} agrees with the already established notion for discrete groups; cf.~\cite[Definition 2.23 \& Lemma 2.24]{sauer-thesis}, where the terminology \emph{bounded} measure equivalence is used.

\end{remark}

Our definition of uniform measure equivalence is motivated by the fact that two 
cocompact, unimodular, closed subgroups of the same  lcsc group should be uniformly measure equivalent. {This fact was also among the} motivations for introducing uniform measure equivalence in the setting of discrete groups.  Before showing that this is indeed the case, we prove a  lemma that describes the relationship between coarse structures and Haar measures of an lcsc group and its homogeneous spaces. For its proof we need the existence of certain {Borel sections} and since these will appear in a number of instances in the sequel we single this out in form of the following theorem, definition and example.

\begin{thm}[Arsenin and Kunugui. See {\cite[Theorem (18.18)]{kechris-book}}]
  \label{thm:borel-section}
  If $X$ is a standard Borel space and $Y$ is a Polish space and $P\subset X\times Y$ is a Borel set for which $P_x:=\{y\in Y \mid (x,y)\in P\}$ is $\sigma$-compact for every $x\in X$, then the image $\pi_X(P)$ under the projection $\pi_X\colon X\times Y\to X$ is Borel in $X$ and there exists a Borel function $s\colon \pi_X(P) \to Y$ with the property that $(x,s(x))\in P$ for all $x\in \pi_X(P)$.
\end{thm}

\begin{defi}
  Let $X$ and $Y$ be  standard Borel spaces and $P\subset X\times Y$ a Borel set such that $\pi_X(P)$ is Borel.  A Borel function $s \colon \pi_X(P) \to Y$ such that that $(x,s(x))\in P$ for all $x\in \pi_X(P)$ is called a \emph{Borel section} for $P$.
\end{defi}

\begin{ex}
  If $H \leq G$ is a closed subgroup of an lcsc group, we obtain {Borel sections} $G/H \rightarrow G$ and $H \backslash G \rightarrow G$.  If furthermore, $H$ is cocompact in $G$, then there is a compact subset $K \subset G$ that maps surjectively onto $G/H$ and $H \backslash G$ 
and the {sections} may be chosen to have their image in $K$. More generally, if an lcsc group $G$ acts continuously, properly and cocompactly on an lcsc Hausdorff space $\Omega$, then there exists a {Borel section} $\Omega/G \to \Omega$ which is \emph{bounded}, in the sense that it takes values in a compact set.
\end{ex}


\begin{lemma}
  \label{lem:group-homomogeneous-space-coarse-measure-structure}
  Let $H$ be a closed subgroup of an lcsc group $G$.
  Assume that $G/H$ and $H \backslash G$ carry  $G$-invariant measures $\lambda_{G/H}$ and $\rho_{H \backslash G}$, respectively.  If $s\colon G/ H \rightarrow G$ is a Borel section, then the push-forward of $\lambda_{G/H} \times \lambda_H$ along the map $G / H \times H \rightarrow G: (x,h) \mapsto s(x)h$ is a left Haar measure of $G$.  Similarly, if $s\colon H \backslash G \rightarrow G$ is a Borel section, then the push-forward of $\rho_H \times \rho_{H \backslash G}$ along $H \times H \backslash G \rightarrow G: (h,x) \mapsto h s(x)$ is a right Haar measure for $G$.
\end{lemma}
\begin{proof}
  We only prove the cases of right cosets.  Let $s\colon H \backslash G \rightarrow G$ be a Borel section for $G \rightarrow H \backslash G$ and  let $r\colon G \rightarrow H$ be the retract defined by the requirement that $r(g) s(Hg)  = g$ for all $g \in G$.  Define $i\colon H \times H \backslash G \rightarrow G$ by $i(h, Hg) := h s(Hg)$ and note that $i$ is a {Borel} isomorphism with inverse $i^{-1}(g) = (r(g), Hg)$, since $r(g) s(Hg) = g$ for all $g \in G$ and
\begin{align*}
  (r(h s(Hg)), H h s(Hg))
  & =
  (hs(Hg) s(Hs(Hg))^{-1}, Hs(Hg)) \\
  & = 
  (h s(Hg) s(Hg)^{-1}, Hg) \\
  & =
  (h, Hg)
\end{align*}
for all $(h, Hg) \in H \times H \backslash G$. \\
We now assume that $H \backslash G$ admits a $G$-invariant measure $\rho_{H \backslash G}$ and denote by $s\colon H \backslash G \rightarrow G$  a Borel section.  Note that  $\rho_{H \backslash G}$ is unique up to scaling by \cite[Corollary B.1.7]{BHV}.
Let $f \in \mathrm{C}_{\mathrm c}(G)$.  Then
  \begin{align*}
    \int_G f(g) \mathrm{d} i_*(\rho_H \times \rho_{H \backslash G}) (g)
    & =
      \int_{H \times H \backslash G} (f \circ i) (h,Hg) \mathrm{d} (\rho_H \times \rho_{H \backslash G})(h,Hg)  \\
    & =
    \int_{H \backslash G} \int_H f(h s(H g)) \mathrm{d} \rho_H(h) \mathrm{d} \rho_{H \backslash G}(Hg) \\
    & =
    \int_G f(g) \mathrm{d} \rho_G(g),
  \end{align*}
  were that last equality follows from \cite[Corollary B.1.7]{BHV} for a suitable choice of right Haar measure $\rho_G$.  This finishes the proof.
\end{proof}

\begin{proposition}
  \label{prop:cocompact-ume}
 Let $H_1$ and $H_2 $ be two cocompact,  unimodular, closed subgroups of an lcsc group $G$.  Then $H_1$ and $H_2$ are uniformly measure equivalent.
\end{proposition}
\begin{proof}
  By Corollary B.1.8 of \cite{BHV}, the group $G$ is unimodular and $H_1, H_2 \leq G$ have finite covolume; i.e.~there is a finite left-invariant measure $\lambda_{G_2/H}$ on $G / H_2$ and a finite right-invariant measure $\rho_{H_1 \backslash G}$ on $H_1 \backslash G$.
  After a choice of bounded Borel sections $s_1\colon   H_1 \backslash G \rightarrow G$ and $s_2\colon   G/H_2 \rightarrow G$, Lemma \ref{lem:group-homomogeneous-space-coarse-measure-structure} says that the map $i\colon   H_1 \times H_1 \backslash G \rightarrow G$ defined by $i(h, x) = h s_1(x)$ is a Borel isomorphism
  such that $i_*(\rho_{H_1} \times \rho_{H_1 \backslash G}) = \rho_G$ for a right Haar measure $\rho_G$.  Similarly, the map $j\colon   G/H_2 \times H_2 \rightarrow G$ defined by  $j(x,h) = s_2(x)h$ is a Borel isomorphism which satisfies $j_*(\lambda_{G/H_2} \times \lambda_{H_2}) = \lambda_G$ for a left Haar measure $\lambda_G$.  Since $H_1,H_2$ and $G$ are unimodular their left and right Haar measure agree, so by fixing the (left) Haar measure $\lambda_G$ on $G$ and rescaling the measure $\rho_{H_1\backslash G}$ suitably, we obtain that $(G,\lambda_G)$ is a strict measure equivalence coupling between $H_1$ and $H_2$ when endowed with the action $(h_1,h_2).g:=h_1gh_2^{-1}$ and maps
  \begin{align*}
  i&\colon H_1\times H_1 \backslash G \To G, \ (h,y)\mapsto i(h,y),\\
  j'&\colon H_2 \times G/H_2 \To G, \  (h,x)\mapsto j(x,h^{-1}).
  \end{align*}
We are therefore done if we can show that $i^{-1}\circ j'$ and $j'^{-1}\circ i$ are proper maps. To this end, fix a compact subset $D\subset H_2$ and note that since $s_2$ is bounded there exists a compact set $K\subset G$ such that
\[
j'(D\times G/H_2 ):=s_2(G/H_2).D^{-1}\subset K.
\]
Now, 
\[
i^{-1}(K)=\{(h,y)\in H_1\times H_1 \backslash G\mid hs_1(y)\in K \} \subset \{(h,y)\in H_1\times H_1\backslash G\mid h\in Ks_1(H_1\backslash G)^{-1} \}
\]
so that
\[
i^{-1}\circ j'(D\times  G/H_2) \subset (H_1 \cap Ks_1(H_1\backslash G)^{-1})\times Y,
\]
and since $s_1$ is bounded this shows that $i^{-1}\circ j'$ is proper. Properness of $j'^{-1}\circ i$ follows by a similar argument. 
\end{proof}

For finitely generated groups, it was shown in \cite{sauer-thesis} that {uniform measure equivalence} implies quasi-isometry,  and the following proposition shows that this result extends to {lcsc groups}. The proof {follows that in} \cite{sauer-thesis}, with a few additional technicalities stemming from the more general topological setting. They key to the result is the following lemma. 

\begin{lemma}\label{lem:set-coupling-to-qi}
Let $G$ and $H$ be two lcsc groups. If $G \times H$ acts on a set $\Omega$ and there exists a subset {$Z \subset \Omega$} for which 
\begin{itemize}
\item[(i)] $G.Z=\Omega=H.Z$
\item[(ii)] $\lbrace g \in G \colon g.Z \cap Z \neq \emptyset \rbrace$ and $\lbrace h \in H \colon h.Z \cap Z \neq \emptyset \rbrace$ are precompact;
\item[(iii)] for every compact subset {$K \subset G$} there exists a compact subset {$L_K \subset H$} such that {$K.Z \subset L_K.Z$}; and for every compact subset {$L \subset H$} there exists a compact subset {$K_L \subset G$} such that {$L.Z \subset K_L.Z$},  
\end{itemize}
then $G$ and $H$ are coarsely equivalent.
\end{lemma}
\begin{proof}
Fix a compatible, proper, left-invariant metrics $d_G$ and $d_H$ on $G$ and $H$, respectively and, for $E\subset H$ denote  $\sup \lbrace d_H(e_H,a) \colon a \in E \rbrace$ by $\ell_H(E)$ and similarly for $G$. 
Fix $z_0 \in Z$ and write $Z_H := \lbrace h \in H \colon h.Z \cap Z \neq \emptyset\rbrace$, noting that $K_H$ is bounded with respect to $d_H$ so that  $b := \ell_H(Z_H) < \infty$. We now proceed with the actual proof.
By (iii) and the axiom of choice there exists a function $f \colon G \rightarrow H$ for which $g^{-1}.z_0 \in  f(g).Z$. We claim that $f \colon (G,d_G) \rightarrow (H,d_H)$ is  a coarse equivalence. We first prove that $f(G)$ is cobounded. Towards this, fix $h \in H$. Then $h^{-1}.z_0 \in g.Z$ for some $g \in G$ by (i) and since the actions commute $g^{-1}.z_0 \in h.Z$. Since also $g^{-1}.z_0 \in f(g).Z$ it follows that $h^{-1}f(g) \in Z_H$ and hence $d_H(h,f(g)) = d_H(e_H,h^{-1}f(g)) \leq b$ from which it follows that $f$ has cobounded image. By properness of $d_G$ and (iii), for each $r\geq 0$ let {$L(r) \subset H$} denote a compact set for which {$\overline{B}(e_G,r).Z \subset L(r).Z$}. Clearly this can be done so that {$L(r_1) \subset L(r_2)$} when $r_1 \leq r_2$. We now prove that $f$ is a coarse embedding. Towards this, let $g,g' \in G$ and put $r :=d_G(g,g')$. 
Since ${g'}^{-1}.z_0 \in f(g').Z$ and the actions commute ${g'}^{-1}f(g)^{-1}.z_0 \in f(g)^{-1}f(g').Z$; on the other hand
\begin{align}
{g'}^{-1}f(g)^{-1}.z_0 &= {g'}^{-1}(gg^{-1})f(g)^{-1}.z_0 = {g'}^{-1}g f(g)^{-1}g^{-1}.z_0 \nonumber \\ &\in ({g'}^{-1}g)f(g)^{-1}f(g).Z = ({g'}^{-1}g).Z \nonumber \\ &{\subset} \overline{B}(e_G,r).Z \subseteq L(r).Z \nonumber
\end{align}
so  $(f(g)^{-1}f(g').Z) \cap L(r).Z \neq \emptyset$ and hence  $f(g')^{-1}f(g)h \in Z_H$ for for some $h \in L(r)$. Since $Z_H = Z_H^{-1}$ it follows that $f(g) ^{-1}f(g') \in L(r).Z_H$ and by the triangle inequality and left-invariance of $d_H$
\begin{align}
d_H(f(g),f(g')) &= d_H(e_H, f(g)^{-1}f(g')) \leq \ell_H(L(r).Z_H) \nonumber \leq \ell_H(L(r)) + b. \nonumber 
\end{align}
To see that $f$ is a coarse embedding we need to first show that if $(g_n)$ and $(g_n')$ are sequences in $G$ and $\lim_{n \rightarrow \infty}d_H(f(g_n),f(g_n')) = \infty$ then $\lim_{n \rightarrow \infty}d_G(g_n,g_n')= \infty$. So suppose $\lim_{n \rightarrow \infty}d_G(f(g_n),f(g_n'))= \infty$. Then by the above we have that $\lim_{n \rightarrow \infty}\ell_H(L(d_G(g_n,g_n'))) = \infty$.
Towards a contradiction, suppose that $ d_G(g_n, g_n')$ does not diverge towards infinity, and hence that there exists $M>0$ and a subsequences $(g_{n_k}), (g_{n_k}')$ such that $d_G(g_{n_k},g_{n_k}') \leq M$. Then, since $L$ is increasing, 
we have {$L(d_G(g_{n_k},g_{n_k}')) \subset L(M)$}   for all $k \in \mathbb{N}$ and hence that $\sup_{k} \ell_H(L(d_G(g_{n_k},g_{n_k}')))<\infty$, but this contradicts the fact $\lim_{n \rightarrow \infty}\ell_H(L(d_G(g_n,g_n'))) = \infty$.
  Thus $\lim_{n \rightarrow \infty}d_G(g_n,g_n') = \infty$. It remains to prove that if $\lim_{n \rightarrow \infty}d_G(g_n,g_n')= \infty$ then $\lim_{n \rightarrow \infty}d_H(f(g_n),f(g_n')) = \infty$.  By properness of $d_H$ and (iii) let {$D(r) \subset G$} denote a compact set for which {$\overline{B}(e_H,r).Z \subset D(r).Z$} and {$D(r_1) \subset D(r_2)$} whenever $r_1 \leq r_2$. Consider now $f(g),f(g') \in H$ and let put $r:=d_H(f(g),f(g'))=r$ and note, as above, that ${g'}^{-1}. f(g)^{-1}.z_0 \in f(g)^{-1}f(g').Z$ so ${g'}^{-1}. f(g)^{-1}.z_0 \in D(r).Z$. On the other hand, ${g'}^{-1}. f(g)^{-1}.z_0 = {g'}^{-1}.(gg^{-1})f(g)^{-1}.z_0 = {g'}^{-1}gf(g)^{-1}g^{-1}.z_{0} \in {g'}^{-1}g.Z$ so ${g'}^{-1}g.Z \cap D(r).Z \neq \emptyset$ and hence ${g'}^{-1}g \in D(r).Z$. So as previously $$d_G(g,g') \leq \ell_G(D(r)) + b'$$   
where $b'=\ell_G(Z_G) < \infty$ and now a similar argument as before shows that if $\lim_{n \rightarrow \infty}d_G(g_n,g_n')= \infty$ then $\lim_{n \rightarrow \infty}d_H(f(g_n),f(g_n')) = \infty$.
\end{proof}


\begin{prop}\label{prop:UME-implies-QI}
If $G$ and $H$ are uniformly measure equivalent, unimodular, lcsc groups then $G$ and $H$ are coarsely equivalent.
\end{prop}
\begin{proof}
Let $(\Omega, \eta, X,\mu, Y,\nu, i,j)$ be a strict, uniform, measure equivalence coupling and
define {$Z := i(\{e_G\}\times Y )\cup j(\{e_H\}\times X)\subset \Omega$}. We now consider $\Omega$ simply as a set with commuting, set-theoretical actions of $G$ and $H$ and aim at proving that $Z$ satisfies the assumptions (i)-(iii) in  Lemma \ref{lem:set-coupling-to-qi}.  Condition (i) is trivially satisfied. To verify (ii), we need to prove that $\{h\in H\mid h.Z\cap Z\neq \emptyset\}$ is precompact in $H$ (the corresponding statement for $G$ then follows by symmetry). Since $j^{-1}\circ i$ is proper, there exists a compact set $D\subset H$ such that $Z\subset j(D\times X)$ and we therefore have
\[
\{h\in H\mid h.Z\cap Z\neq \emptyset\}\subset \{h\in H\mid h.D\cap D\neq \emptyset \},
\]
and the latter is compact since $H$ acts properly on itself. Lastly we need to see that {(iii)} is satisfied, which will follow from the cocycles being locally bounded; more precisely, if $C\subset G$ is compact then, by Lemma \ref{lem:cocycles-loc-bd}, the set $D:=\{\omega_H(g,x)\mid g\in C, x\in X\}$ is precompact and we have
\[
C.j(\{e_H\}\times X)\subset j(D \times X)=D .j(\{e_H\}\times X).
\]
Moreover, we have $C.i(\{e_G\}\times Y)=i(C\times Y)\subset j(D'\times X)$ for some compact set $D'\subset H$ and hence the closure, $D''$, of $D \cup D'$ satisfies
\[
C.Z\subset j(D''\times X)=D''.j(\{e\}\times X)\subset D''.Z. \qedhere
\]

\end{proof}

 For discrete groups, Gromov's dynamic criterion for quasi-isometry \cite[0.2.C$_2'$]{gromov-asymptotic-invariants} (see also \cite{shalom-harmonic-analysis, sauer-thesis}) plays a prominent role, as it allows one to treat quasi-isometry within a purely topological framework, and in \cite{bader-rosendal} Bader and Rosendal generalized this to the locally compact setting by proving the following result.
\begin{thm}[{\cite[Theorem 1]{bader-rosendal}}]
Two lcsc groups, $G$ and $H$, are coarsely equivalent if and only if there exists a locally compact Hausdorff space $\Omega$ with commuting, continuous,  proper, cocompact actions of $G$ and H.
\end{thm}
Recall that an action $G\curvearrowright \Omega$ on a locally compact Hausdorff space $\Omega$ is said to be \emph{continuous} if the action map $G\times \Omega \to \Omega$ is continuous, \emph{cocompact} if it is continuous and there exists a compact subset {$C_0\subset \Omega$} such that $G.C_0=\Omega$ and \emph{proper} if the action map $G\times \Omega \to \Omega \times \Omega$, $(g,t)\mapsto (g.t,t)$, is proper; i.e.~the inverse image of any compact set is compact.
A locally compact Hausdorff space $\Omega$ with commuting, continuous, proper, cocompact actions of $G$ and $H$ is called a \emph{topological coupling between $G$ and $H$}. \\

The main result in this section is the following theorem showing that, in analogy with the discrete case \cite[Theorem 2.38]{sauer-thesis}, uniform measure equivalence agrees with coarse equivalence on the class of amenable unimodular lcsc groups.

\begin{thm}
  \label{thm:uniform-me-amenable-groups}
  Let $G$ and $H$ be amenable,  unimodular, lcsc groups.  Then $G$ and $H$ are coarsely equivalent if and only if they are uniformly measure equivalent.
\end{thm}
For the proof of Theorem \ref{thm:uniform-me-amenable-groups}, we {need the observation that the topological coupling by} {Bader--Rosendal is second countable.} 
We remark that the only novelty in Lemma \ref{lem:sc-coupling} compared to \cite[Theorem 1]{bader-rosendal} is the fact that the  topological space witnessing the coarse equivalence can be chosen to be second countable.  This technicality, however, will allow us to work exclusively within the class of standard Borel spaces. 

\begin{lem}\label{lem:sc-coupling}
If $G$ and $H$ are coarsely equivalent lcsc  groups then there exists an lcsc Hausdorff space $\Omega$ with commuting, continuous,  proper, cocompact actions of $G$ and $H$.
\end{lem}
{For the proof, recall from \cite{Michael2} that a collection of} {(not necessarily open) subsets $\mathcal{P} \subset  \mathcal{P}(Y)$ of a topological space $Y$ is a \emph{pseudobase} if whenever} {$K \subset  U$ where $U \subset  Y$ is open and $K \subset  Y$ is compact there exists $P \in \mathcal{P}$ for which $K \subset  P \subset  U$.} {The set $\mathcal{P}$ is a \emph{network} for $Y$ if whenever $x \in U$ where $U \subset  Y$ is open there exists $P \in \mathcal{P}$ for} {which $x \in P \subset  U$. A regular Hausdorff space $Y$ is an \emph{$\aleph_0$-space} if it has a countable pseudobase} {and \emph{cosmic} if it has a countable network.}
\begin{proof}
  From the proof of \cite[Theorem 1]{bader-rosendal}, {there is a topological coupling $\Omega$ between $G$ and $H$} {that is a locally compact closed subspace of $C(H,X)$}, where $X$ is a certain {$G$-}invariant subset of the separable Banach space $L^1(G,\lambda_G)$ and $C(H,X)$ is the space of continuous functions from $H$ to $X$ {with} the topology of pointwise convergence in 1-norm.
We now prove that $\Omega$ is second countable with respect to the subspace topology. 
As $H$ and $X$ are second countable, regular, Hausdorff spaces they are $\aleph_0$-spaces and $C(H,X)$, with the topology of pointwise convergence,  is cosmic \cite[Proposition 10.4]{Michael2}. Now $\Omega \subset  C(H,X)$ 
so $\Omega$ is cosmic as well \cite[Condition (E)]{Michael2} (for cosmic spaces, page 994). Since $\Omega$ is locally compact {\cite[Proof of Theorem 1, Claim 4]{bader-rosendal}} and every locally compact cosmic space is separable and metrizable \cite[Condition (C)]{Michael2} (for cosmic spaces, page 994), $\Omega$ is second countable. 
\end{proof}

\begin{proof}[Proof of Theorem \ref{thm:uniform-me-amenable-groups}]
  By Proposition \ref{prop:UME-implies-QI}, we know that $G$ and $H$ are coarsely equivalent if they are uniformly measure equivalent, so we have to show the converse.  If $G$ and $H$ are coarsely equivalent  then, by \cite[Theorem 1]{bader-rosendal} and Lemma \ref{lem:sc-coupling}, they admit an lcsc topological coupling $\Omega$.  There exists is a free action of $G \times H$ on a compact metrizable (and hence second countable) space by \cite[Proposition 5.3]{adams-stuck}, so by amplifying $\Omega$ with such a space we may assume that the $G\times H$-action is free as well as proper.
 We now show that $\Omega$ can be endowed with a measure turning it into a uniform measure equivalence coupling. Let $\pi_X\colon  \Omega \rightarrow \Omega/H =: X$ and $\pi_Y \colon   \Omega \rightarrow  \Omega / G =: Y$ denote the quotient maps.  Since the actions of $G$ and $H$ on $\Omega$ are proper and cocompact, $X$ and $Y$ are compact Hausdorff spaces and since $\Omega$ is separable so are $X$ and $Y$, and hence they are metrizable.  
 By Theorem \ref{thm:borel-section}, there exist bounded Borel sections $s_X\colon X\to \Omega$ and $s_Y\colon Y \to \Omega$ for $\pi_X$ and $\pi_Y$, respectively.  Since $\Omega$ is a free topological coupling, we obtain Borel isomorphisms
\begin{gather*}
  i\colon G \times Y \rightarrow \Omega: \, (g,y) \mapsto g. s_Y(y), \\
  j\colon H \times X \rightarrow \Omega: \, (h,x) \mapsto h. s_X(x)
  \text{.}
\end{gather*}
Since $H$ acts freely on $\Omega$ a map $r_X\colon \Omega \to H$ is defined by the relation $r_X(t)\pi_X(s_X(t))=t$ and similarly we obtain a map $r_Y\colon \Omega \to G$. We have $j^{-1}(t)=(r_H(t), \pi_X(t))$ so that $j^{-1}\circ i (g,y)=(r_H(g.s_Y(y)), \pi_X(g.s_Y(y)))$, and similarly for $i^{-1}\circ j$

We now show that $i^{-1}\circ j$ and $j^{-1}\circ i$ are proper, and by  symmetry of the situation it suffices to treat $i^{-1}\circ j$. So, for a given compact subset $D\subset H$ we need to show that there exists a compact subset $C\subset G$ such that $i^{-1}\circ j(D\times X)\subset C\times Y$.  Since the action $H\curvearrowright \Omega$ is continuous and $s_X(X)$ is precompact there exists a compact subset $K\subset \Omega$ such that $D.s_X(X)\subset K$.  Then
$
i^{-1}(K)=\{(g, y)\in G\times Y\mid {g.s_Y(y)}\in K\},
$
and since the action $G\curvearrowright \Omega $ is proper and $s_Y$ is bounded, the set $\{g\in G \mid g.s_Y(Y)\cap K\neq \emptyset \}$ is precompact, and its closure, $C$, therefore satisfies that
\[
i^{-1}\circ j (D\times X)\subset i^{-1}(K)\subset C\times Y,
\]
as desired.
We now endow $G \times Y$ with an $H$-action, by pulling back the $H$-action from $\Omega$, and similarly, we pull back the $G$-action on $\Omega$ to a $G$-action on $H \times X$.  By $G$- and $H$-equivariance of $i$ and $j$, respectively, we obtain free  Borel actions of $G \times H$ on $G \times Y$ and on $H \times X$ with respect to which $i$ and $j$ are now, by design, $G\times H$-equivariant.  Hence, there is 
 a Borel $H$-action on $Y$ defined by the composition
  \begin{equation*}
    H \times Y
    \xrightarrow{(h,y)\mapsto (h,e_G,y)}
    H \times G \times Y
    \xrightarrow{\id \times i}
    H \times \Omega
    \longrightarrow
    \Omega
    \stackrel{\pi_Y}{\longrightarrow}
    Y,
  \end{equation*}
and a Borel $G$-action on $X$ defined by the composition
  \begin{equation*}
    G \times X
   \xrightarrow{(g,x)\mapsto (g,e_H,x)}
    G \times H \times X
      \xrightarrow{\id \times j}
    G \times \Omega
    \longrightarrow
    \Omega
    \stackrel{\pi_X}{\longrightarrow}
    X.
  \end{equation*}
We next note that the $H$-action on $Y$ and the $G$-action on $X$ are continuous, since they agree with the actions induced by the $G \times H$-action on $\Omega$ via the continuous projections $\pi_Y$ and $\pi_X$, respectively.   Indeed, for $y \in Y$, we have $h.y = \pi_Y(h .i(e, y)) = \pi_Y(h. s_Y(y))$ so that the action agrees with the natural continuous action $H\curvearrowright \Omega/G$. So $H \curvearrowright Y$ is a continuous action by homeomorphisms on a compact metrizable space,  and since $H$ is amenable, there is an $H$-invariant probability measure $\nu$ on $Y$ (cf.~\cite[Chapter 4]{zimmer-book}).  Let $\lambda_G$ be a Haar measure on $G$ and put $\eta := i_*(\lambda_G \times \nu)$.  Note that the action $H\curvearrowright G\times Y$ is given by
\begin{align*}
h.(g,y) &= i^{-1}(h.i(g,y))\\
&=i^{-1}(hg.s_Y(y))\\
&=g.i^{-1}(h.s_Y(y))\\
&=g.(r_Y(h.s_Y(y)),\pi_Y(h.s_Y(y)) )\\
&=(gr_Y(h.s_Y(y)),h.y )
\end{align*}
and since $G$ is unimodular and $\nu$ is $H$-invariant it follows that $\lambda_G\times \nu$ is $G\times H$-invariant.
Since both $i$ and $j$ are $H$-equivariant, the measure $\rho:=(j^{-1})_* \eta= (j^{-1}\circ i)_*(\lambda_G\times \nu)$ on $H \times X$ is therefore invariant for the action of $H$ on the left leg. For each Borel set $B\subset X$ we can define a left-invariant Borel measure $\rho_B$ on $H$ by setting $\rho_B(U):=\rho(U\times B)$ and since $j^{-1}\circ i$ is proper and $\lambda_G\times \nu$ is finite on sets of the form $K\times Y$ with $K\subset G$ compact, we have that $\rho_B$ is finite on compact subsets and hence there exists $c_B\in [0,\infty[$ so that $\rho_B=c_B\lambda_H$. Setting $\mu(B):=c_B$ defines a finite measure on $X$ with the property that $\rho(A\times B)=\lambda_H(A)\mu(B)$ for all Borel sets $A\subset H$ and $B\subset X$, and by uniqueness of the product measure we conclude that $\rho=\lambda_H\times \mu$. Since we already saw that  $j^{-1}\circ i $ and $i^{-1}\circ j$ are proper this shows that $(\Omega,\eta)$ is a strict uniform measure equivalence coupling. 
\end{proof}

As a consequence of the proof just given, we obtain the following slightly more specific statement.
\begin{por}
If $G$ and $H$ are coarsely equivalent, amenable, {unimodular lcsc groups} then they admit a free, {Hausdorff}, lcsc topological coupling and any such coupling admits a Borel measure with respect to which it is a strict uniform measure coupling.
\end{por}

\def\cprime{$'$} \def\cprime{$'$}

\end{document}